\documentclass[oneside]{amsart}

\usepackage{amsmath,amsthm,amsfonts,amssymb}
\usepackage[all]{xy}
\usepackage{color}
\usepackage{xparse}
\usepackage{aliascnt}
\usepackage[pagebackref,colorlinks,citecolor=blue,linkcolor=blue,urlcolor=blue,filecolor=blue]{hyperref}
\usepackage{bbm}
\usepackage{enumerate}

\usepackage{tikz}
\usetikzlibrary{shapes.geometric}
\tikzset{
v/.style={draw, fill, circle, minimum size=1.5mm, inner sep=0},
b/.style={draw , regular polygon,regular polygon sides=4, minimum size=1.5mm, inner sep=.5mm},
e/.style={very thick},
vs/.style={draw, fill, circle, minimum size=1mm, inner sep=0},
bs/.style={draw,  regular polygon,regular polygon sides=4, minimum size=2mm, inner sep=0mm},
es/.style={thick}
}
\newlength{\nodeheight}
\newlength{\nodewidth}
\usetikzlibrary{arrows,matrix,positioning}

\numberwithin{thmcounter}{section}
\newaliascnt{thmauto}{thmcounter}

\newaliascnt{Defauto}{thmcounter}

\newaliascnt{exauto}{thmcounter}

\newaliascnt{lemauto}{thmcounter}

\newaliascnt{propauto}{thmcounter}

\newaliascnt{corauto}{thmcounter}

\newaliascnt{remauto}{thmcounter}

\newaliascnt{clmauto}{thmcounter}

\newtheorem{atheorem}{Theorem}
\newtheorem{acor}[atheorem]{Corollary}

\newtheorem{thm}[thmauto]{Theorem}
\newtheorem{lem}[lemauto]{Lemma}
\newtheorem{prop}[propauto]{Proposition}

\theoremstyle{definition}
\newtheorem{Def}[Defauto]{Definition}
\newtheorem{defn}[Defauto]{Definition}
\newtheorem{ex}[exauto]{Example}
\newtheorem{rem}[remauto]{Remark}

\numberwithin{equation}{section}

\DeclareMathOperator{\Tor}{Tor}
\DeclareMathOperator{\Ext}{Ext}
\providecommand{\id}{\ensuremath\mathrm{id}}

\DeclareMathOperator{\TL}{TL}

\DeclareMathOperator{\Br}{Br}

\providecommand{\fS}{\ensuremath\mathfrak{S}}

\providecommand{\CBr}{\ensuremath\mathcal C_{\Br}}

\providecommand{\tens}[1][{}]{\otimes_{#1}}

\renewcommand{\t}{\mathbbm{1}}
\newcommand{\bbC}{\mathbb{C}}

\title{The homology of the Brauer algebras}
\author{Rachael Boyd}
\address{Max Planck Institute for Mathematics, Bonn}
\email{rachaelboyd@mpim-bonn.mpg.de}
\author{Richard Hepworth}
\address{Institute of Mathematics, University of Aberdeen}
\email{r.hepworth@abdn.ac.uk}
\author{Peter Patzt}
\thanks{Peter Patzt was supported by the Danish National Research Foundation 
	through the Copenhagen Centre for Geometry and Topology (DNRF151) and 
	the European Research Council under the European Union’s Seventh 
	Framework Programme ERC Grant agreement ERC StG 716424 - CASe, PI Karim 
	Adiprasito.}
\address{Centre for Geometry and Topology, University of Copenhagen\newline Department of Mathematics, University of Oklahoma}
\email{ppatzt@ou.edu}


\begin{document}

\maketitle
\begin{abstract}
This paper investigates the homology of the Brauer algebras, interpreted as appropriate $\Tor$-groups, and shows that it is closely related to the homology of the symmetric group.
Our main results show that when the defining parameter~$\delta$ of the Brauer algebra is invertible, then the homology of the Brauer algebra is isomorphic to the homology of the symmetric group, and that when~$\delta$ is not invertible, this isomorphism still holds in a range of degrees that increases with $n$.
\end{abstract}

\setcounter{tocdepth}{1}
\tableofcontents

\section{Introduction}\label{Section: Introduction}

In this paper we study the homology of the Brauer algebra $\Br_n(R,\delta)$, interpreted as appropriate $\Tor$ groups.
We show that it is isomorphic to the homology of the symmetric group if the parameter $\delta$ is invertible, and that this holds in a range of degrees that increases with $n$ for general $\delta$.
Our methods stem largely from the theory of homological stability for families of groups, but with important and novel adaptations required for the algebraic setting.

\subsection{Homological stability and representation stability}

Homological stability is a long-studied notion that concerns the ho\-mo\-lo\-gy and cohomology of sequences of groups or spaces. 
A sequence $\{X_n\}$ of topological spaces or groups, equipped with maps $X_{n-1} \to X_{n}$, is said to satisfy homological stability if for all $i$ the induced maps $H_i(X_{n-1}) \to H_i(X_{n})$ are isomorphisms, when~$n$ is sufficently large compared to $i$.
Homological stability holds for symmetric groups \cite{N}, un\-or\-dered configuration spaces \cite{Arnold,McDuff,C}, the automorphism groups of free groups \cite{H,HVRational,HV}, mapping class groups \cite{Har,WahlUnoriented,HatcherWahl}, diffeomorphism groups of high-dimensional manifolds~\cite{GRW}, and general linear groups~\cite{Maazen,Charney,vdK}, among others.
Randal-Williams and Wahl~\cite{RW} recently gave a unified approach to studying homological stability for sequences of groups that assemble into a \emph{homogeneous category}.
Their theory covers many of the known examples, and includes a systematic construction of the chain complexes that are commonly used to prove homological stability.

In the past 10 years the closely-related concept of \emph{representation stability}~\cite{CF} has emerged, with the aim of proving stability phenomena when the terms~$X_n$ in the sequence are not simply modules or vector spaces, but representations of some sequence of groups.
In~\cite{PatztRepstab}, the third author built a \emph{stability category}, of the kind appearing in~\cite{RW}, for sequences of \emph{diagram  algebras}: the Brauer algebras, the Temperley-Lieb algebras, and the partition algebras. For the Brauer and partition algebras it is shown the associated categories are symmetric monoidal. Furthermore, for all three algebras the notion of representation stability from~\cite{CF} is generalised and proved for sequences of semisimple representations over these algebras.

Recent work of the first two authors has demonstrated that the techniques of homological stability can be applied to families of \emph{algebras}, where homology and cohomology are interpreted as appropriate $\Tor$ and $\Ext$ groups: for Iwahori-Hecke algebras of type $A$~\cite{HepworthIH} and for Temperley-Lieb algebras~\cite{HepworthBoydStability}.
In~\cite{HepworthIH} an adaptation of the classic techniques was sufficient, but in~\cite{HepworthBoydStability} it was necessary to overcome a new technical hurdle, the failure of a flatness condition required for Shapiro's lemma to hold. 
To resolve this the authors developed a new technique, which they call~\emph{inductive resolutions}. 
The question of whether this technique could be adapted to prove homological stability results for other diagrammatic algebras, such as the Brauer and partition algebras, was posed in the introduction of~\cite{HepworthBoydStability}. In this work we answer that question in the affirmative.

\subsection{Brauer algebras}

The Brauer algebra was introduced by Brauer \cite{B} to study representations of the orthogonal and the symplectic groups. 
It also arises in statistical mechanics, and, together with its variants, has become an important object in representation theory.
Given a commutative ring $R$ and a fixed element $\delta\in R$, the Brauer algebra $\Br_n = \Br_n(R,\delta)$ is the free $R$-module over the basis of all perfect matchings of the union of the sets $[-n]=\{-n,\dots,-1\}$ and $[n]= \{1,\dots,n\}$. We visualize such a matching as a \emph{diagram} by placing $n$ nodes on the left for the negative numbers, $n$ nodes on the right for the positive numbers, and connecting two nodes if they are matched. An example is shown below for the perfect matching $\{\{-1,-3\},\{-2,-4\},\{-5,3\},\{1,5\},\{2,4\}\}\in \Br_5$. 
\begin{center}
\begin{tikzpicture}
\fill (0,0)           circle (.75mm) node[left=2pt](a5) {$-1$};
\fill (0,\nodeheight) circle (.75mm) node[left=2pt](a4) {$-2$};
\multiply\nodeheight by 2
\fill (0,\nodeheight) circle (.75mm) node[left=2pt](a3) {$-3$};
\divide\nodeheight by 2
\multiply\nodeheight by 3
\fill (0,\nodeheight) circle (.75mm) node[left=2pt](a2) {$-4$};
\divide\nodeheight by 3
\multiply\nodeheight by 4
\fill (0,\nodeheight) circle (.75mm) node[left=2pt](a1) {$-5$};
\divide\nodeheight by 4

\fill (\nodewidth,0) circle (.75mm) node[right=2pt](b5) {$1$};
\fill (\nodewidth,\nodeheight) circle (.75mm) node[right=2pt](b4) {$2$};
\multiply\nodeheight by 2
\fill (\nodewidth,\nodeheight) circle (.75mm) node[right=2pt](b3) {$3$};
\divide\nodeheight by 2
\multiply\nodeheight by 3
\fill (\nodewidth,\nodeheight)  circle (.75mm) node[right=2pt](b2) {$4$};
\divide\nodeheight by 3
\multiply\nodeheight by 4
\fill (\nodewidth,\nodeheight)  circle (.75mm) node[right=2pt](b1) {$5$};
\divide\nodeheight by 4

\draw[e] (a1) to[out=0, in=180] (b3);
\draw[e] (a2) to[out=0, in=0]   (a4);
\draw[e] (b2) to[out=180, in=180] (b4);
\draw[e] (a5) to[out=0, in=0] (a3);
\draw[e] (b1) to[out=180,in=180] (b5);

\end{tikzpicture}
\end{center}
Observe that a diagram has left-to-left connections if and only if it has right-to-right connections.
Given two diagrams  $b_1$ and $b_2$, to calculate their product $b_1b_2$, we place $b_1$ to the left of $b_2$ so that the nodes align. We obtain a new diagram $b_3$ by removing all loops in the middle and replacing each instance of a loop with multiplication by~$\delta$. For example, if we replace $k$ loops, then $b_1b_2 = \delta^k b_3$. Below we show a visualisation of the multiplication in~$\Br_5$.
\begin{center}

\[
\begin{tikzpicture}[x=1.5cm,y=-.5cm,baseline=-1.05cm]

\node[v] (a1) at (0,0) {};
\node[v] (a2) at (0,1) {};
\node[v] (a3) at (0,2) {};
\node[v] (a4) at (0,3) {};
\node[v] (a5) at (0,4) {};

\node[v] (b1) at (1,0) {};
\node[v] (b2) at (1,1) {};
\node[v] (b3) at (1,2) {};
\node[v] (b4) at (1,3) {};
\node[v] (b5) at (1,4) {};

\draw[e] (a1) to[out=0, in=180] (b3);
\draw[e] (a3) to[out=0, in=0] (a5);
\draw[e] (a2) to[out=0, in=0]   (a4);
\draw[e] (b2) to[out=180, in=180] (b4);
\draw[e] (b1) to[out=180,in=180] (b5);

\end{tikzpicture}
\quad
\cdot
\quad
\begin{tikzpicture}[x=1.5cm,y=-.5cm,baseline=-1.05cm]

\node[v] (b1) at (0,0) {};
\node[v] (b2) at (0,1) {};
\node[v] (b3) at (0,2) {};
\node[v] (b4) at (0,3) {};
\node[v] (b5) at (0,4) {};

\node[v] (c1) at (1,0) {};
\node[v] (c2) at (1,1) {};
\node[v] (c3) at (1,2) {};
\node[v] (c4) at (1,3) {};
\node[v] (c5) at (1,4) {};

\draw[e] (b2) to[out=0,in=0] (b5);
\draw[e] (b3) to[out=0,in=180] (c1);
\draw[e] (b4) to[out=0,in=0] (b1);
\draw[e] (c5) to[out=180,in=180] (c2);
\draw[e] (c3) to[out=180,in=180] (c4);

\end{tikzpicture}
\quad
=
\quad
\begin{tikzpicture}[x=1.5cm,y=-.5cm,baseline=-1.05cm]

\node[v] (a1) at (0,0) {};
\node[v] (a2) at (0,1) {};
\node[v] (a3) at (0,2) {};
\node[v] (a4) at (0,3) {};
\node[v] (a5) at (0,4) {};

\node[v] (b1) at (1,0) {};
\node[v] (b2) at (1,1) {};
\node[v] (b3) at (1,2) {};
\node[v] (b4) at (1,3) {};
\node[v] (b5) at (1,4) {};

\node[v] (c1) at (2,0) {};
\node[v] (c2) at (2,1) {};
\node[v] (c3) at (2,2) {};
\node[v] (c4) at (2,3) {};
\node[v] (c5) at (2,4) {};

\draw[e] (a1) to[out=0, in=180] (b3);
\draw[e] (a3) to[out=0, in=0] (a5);
\draw[e] (a2) to[out=0, in=0]   (a4);
\draw[e] (b2) to[out=180, in=180] (b4);
\draw[e] (b1) to[out=180,in=180] (b5);

\draw[e] (b2) to[out=0,in=0] (b5);
\draw[e] (b3) to[out=0,in=180] (c1);
\draw[e] (b4) to[out=0,in=0] (b1);
\draw[e] (c5) to[out=180,in=180] (c2);
\draw[e] (c3) to[out=180,in=180] (c4);

\end{tikzpicture}
\quad
= \delta^1 \cdot
\quad
\begin{tikzpicture}[x=1.5cm,y=-.5cm,baseline=-1.05cm]

\node[v] (a1) at (0,0) {};
\node[v] (a2) at (0,1) {};
\node[v] (a3) at (0,2) {};
\node[v] (a4) at (0,3) {};
\node[v] (a5) at (0,4) {};

\node[v] (c1) at (1,0) {};
\node[v] (c2) at (1,1) {};
\node[v] (c3) at (1,2) {};
\node[v] (c4) at (1,3) {};
\node[v] (c5) at (1,4) {};

\draw[e] (a1) to[out=0, in=180] (c1);
\draw[e] (a2) to[out=0, in=0]   (a4);
\draw[e] (a5) to[out=0, in=0] (a3);
\draw[e] (c3) to[out=180,in=180] (c4);
\draw[e] (c2) to[out=180,in=180] (c5);
\end{tikzpicture}
\]
\end{center}
The Brauer algebra was introduced as the Schur-Weyl dual of the orthogonal group $O(m)$, in which case $\delta=m$.  It is also the Schur-Weyl dual of the symplectic group $\mathrm{Sp}(2m)$, in which case $\delta=-2m$ (cf.\ \cite{B,HanlonWales}).
It is closely related to other \emph{diagram algebras}, in particular the \emph{Temperley-Lieb algebra}~\cite{TL}, in which the diagrams are planar, and the \emph{partition algebra}~\cite{M91}, in which the matchings on $[-n]\cup [n]$ are replaced with arbitrary partitions.

By the \emph{homology} of the Brauer algebra $\Br_n(R,\delta)$
we mean the $\Tor$ groups
\[
    \Tor_\ast^{\Br_n(R,\delta)}(\t,\t),
\]
where $\t$ denotes the \emph{trivial module} consisting of a copy of $R$ on which a diagram acts as the identity if all nodes on the left are connected to ones on the right, or as $0$ if there are left-to-left and right-to-right connections.
This is completely analogous to the homology of a group, which is defined to be the $\Tor$ groups of the trivial module with itself over the group ring, and is a specific instance of the homology of an arbitrary augmented algebra~\cite[2.4.4]{Benson}.

\subsection{Results}
Diagrams in which every node on the left is connected to a node on the right are called \emph{permutation diagrams}, and are in bijection with elements of the symmetric group $\fS_n$.
This gives rise to inclusion and projection maps
\[
    \iota\colon R\fS_n\longrightarrow \Br_n(R,\delta)
    \qquad\text{and}\qquad
    \pi\colon \Br_n(R,\delta)\longrightarrow R\fS_n
\]
where $\iota$ sends permutations to permutation diagrams, and $\pi$ does the reverse, and sends all remaining diagrams to $0$.
In particular, $\pi\circ \iota$ is the identity map on $R\fS_n$.
There are induced homomorphisms $\iota_\ast$ and $\pi_\ast$ on homology groups for which  $\pi_\ast\circ\iota_\ast$ is again the identity, so that the homology of $\fS_n$ appears as a direct summand of the homology of $\Br_n(R,\delta)$.

\begin{atheorem}\label{thmaa}
    Suppose that $\delta$ is invertible in $R$.
    Then the homology of the Brauer algebra is isomorphic to the homology of the symmetric group:
    \[
        \Tor_\ast^{\Br_n(R,\delta)}(\t,\t)\cong H_\ast(\fS_n;\t).
    \]
    Indeed, the inclusion and projection maps
    \[
        R\fS_n\xrightarrow{\ \iota\ }\Br_n(R,\delta)
       \xrightarrow{\ \pi\ }R\fS_n
    \]
    induce inverse isomorphisms
    \[
        \Tor_\ast^{R\fS_n}(\t,\t)
        \xrightarrow[\cong]{\ \iota_\ast\ }
        \Tor_\ast^{\Br_n(R,\delta)}(\t,\t)
        \xrightarrow[\cong]{\ \pi_\ast \ }
        \Tor_\ast^{R\fS_n}(\t,\t).
    \]
\end{atheorem}

It is well known that $H_1(\fS_2;R)\cong R/2R$, whereas a computation using~\cite[3.1.3]{Wei} shows that $\Tor_1^{\Br_2(R,\delta)}(\t,\t)\cong R/2R\oplus R/\delta R$, with $\iota_\ast$ and $\pi_\ast$ having the evident effect.
Thus the isomorphism of \autoref{thmaa} fails when $\delta$ is not invertible.
Nevertheless, the isomorphism does hold within a range of degrees, and this range of degrees increases with $n$:

\begin{atheorem}\label{thma}
The inclusion map $\iota\colon R\mathfrak S_n \to \Br_n(R,\delta)$ induces a map in homology
\[ \iota_\ast\colon H_i(\mathfrak S_n;\t) \longrightarrow \Tor^{\Br_n(R,\delta)}_i(\t,\t)\]
that is an isomorphism in the range $n \ge 2i+1$.
\end{atheorem}

An immediate consequence of \autoref{thma} is the following corollary.

\begin{acor}
	The Brauer algebras satisfy homological stability, that is the inclusion $\Br_{n-1}(R,\delta)\hookrightarrow \Br_n(R,\delta)$
	induces a map
	\[
	\Tor^{\Br_{n-1}(R,\delta)}_i(\t,\t)
	\longrightarrow
	\Tor^{\Br_{n}(R,\delta)}_i(\t,\t)
	\]
	that is an isomorphism in degrees $n\geq 2i+1$, and this stable range is sharp.
	Furthermore, $\Br_n(R,\delta)$ and $\fS_n$ have the same stable homology:
	\[
	\lim_{n\to\infty}H_\ast(\fS_n;\t)
	\cong
	\lim_{n\to\infty}\Tor_\ast^{\Br_n(R,\delta)}(\t,\t).
	\]
\end{acor}
	
The first part of this corollary follows by combining \autoref{thma} with the corresponding homological stability result for the symmetric groups, for which the stable range is sharp~\cite{N}.
Indeed, stability for the symmetric groups is an important ingredient in our proof of~\autoref{thma}.

For the stable homology, the left hand side of this isomorphism is well known by the Barratt-Priddy-Quillen theorem~\cite{BarrattPriddy,QuillenQ}. This situation is reminiscent of the relationship between the $\fS_n$ and the automorphism groups of free groups $\mathrm{Aut}(F_n)$:
both families satisfy homological stability, there is an inclusion $\fS_n\hookrightarrow\mathrm{Aut}(F_n)$ that induces an isomorphism in the stable range, and they have isomorphic stable homology~\cite{HV,Galatius}.

%Another consequence of \autoref{thma} is that the homology of $\Br_n(R,\delta)$ is independent of $\delta$ in a range of degrees that increases with $n$, and that stably there is no dependence on $\delta$ at all.

It is interesting to ask whether the range $n\geq 2i+1$ in which $H_i(\fS_n;\t)$ and $\Tor_i^{\Br_n(R,\delta)}(\t,\t)$ are isomorphic can be improved, or whether it is sharp.
Sharpness certainly holds in the case $n=2$ by the computation following \autoref{thmaa}, but we do not know what happens for $n\geq 3$.

Consider the case where the ground ring $R$ is the complex numbers $\bbC$. 
Then $R\fS_n$ is semisimple, and so the homology of $\fS_n$ vanishes in positive degrees.
On the other hand $\Br_n(\bbC,\delta)$ is semisimple for all~$n$ only when $\delta$ is not an integer.
In the case that $\delta$ is a \emph{nonzero} integer, \autoref{thmaa} shows that the homology of $\Br_n(\bbC,\delta)$ still vanishes despite the failure of semisimplicity.
In the remaining case $\delta=0$, the homology $\Tor^{\Br_n(\bbC,\delta)}_i(\t,\t)$ vanishes in the stable range $n\geq 2i+1$ by \autoref{thma}, but not more generally, at least in the case $n=2$, and again we do not know what happens when $n\geq 3$.

\subsection{Method of proof}

The proof of \autoref{thmaa} uses an adaptation of the method of \emph{inductive resolutions} introduced in~\cite{HepworthBoydStability} for the Temperley-Lieb algebras.
The method shows that if $J_n\subseteq\Br_n(R,\delta)$ denotes the ideal spanned by all diagrams that have at least one left-to-left connection, then  $\Tor_i^{\Br_n(R,\delta)}(\t,\Br_n(R,\delta)/J_n)$ vanishes for $i>0$, and is $R$ for $i=0$.
This property of the module $\Br_n(R,\delta)/J_n$, combined with the fact that it admits the structure of a \emph{free} right $\fS_n$ module, allows us to complete the proof of \autoref{thmaa} using a novel but elementary homological algebra argument.

The proof of \autoref{thma} begins with the construction of a complex $C_n$ of right $\Br_n(R,\delta)$ modules.
This complex is obtained by adapting the complexes $W_n(X,A)$ of~\cite{RW} to the setting of Brauer algebras using the stability category developed by the third author in~\cite{PatztRepstab}.
We prove that $C_n$ is highly acyclic using an intricate diagrammatical argument that features repeated splittings and filtrations of the complexes involved.

The complexes $W_n(X,A)$ were designed to be used in proofs of homological stability for families of groups, and one can analogously attempt to use the $C_n$ in a proof of homological stability for the $\Br_n(R,\delta)$.
This is our approach, however, a technical hurdle presents itself. 
The resulting spectral sequence features the groups $\Tor_\ast^{\Br_n}(\t,\Br_n\otimes_{\Br_{m}}\t)$.
One would like to identify these with the groups $\Tor_\ast^{\Br_m}(\t,\t)$ using Shapiro's lemma, but that is not possible because $\Br_n$ is not necessarily flat over $\Br_m$.
Instead, a variant of the previously mentioned inductive resolution method is employed, and shows that the above $\Tor$ groups {can} be identified with  $\Tor_\ast^{R\fS_n}(\t,R\fS_n\otimes_{R\fS_m}\t)\cong H_\ast(\fS_m;\t)$.
Combining all of this with homological stability for symmetric groups allows us to complete the proof.

It is interesting to compare our work here with~\cite{HepworthBoydStability}, which gave rather similar results for the Temperley-Lieb algebras.
Oversimplifying, the story of \cite{HepworthBoydStability} is similar to the last two paragraphs, except that instead of finding a ranges of degrees where symmetric groups appear, one finds vanishing ranges.
However, the high-connectivity proof in~\cite{HepworthBoydStability} was much more involved than in the present paper: the symmetric groups do not lie inside the Temperley-Lieb algebras, instead one can only find a homomorphism from the braid group, and this homomorphism is not compatible with the basis of diagrams, so in the end an involved algebraic argument is required.
Also, \cite{HepworthBoydStability} features a sharpness result for the analogue of \autoref{thma}, and we have been unable to find such a sharpness result in this case.

\subsection{Outline}
In \autoref{Section: Background}, we give background on the Brauer algebras and $\Tor$ functors. In \autoref{Section: InductiveResolutions}, we adapt the inductive resolution approach of~\cite{HepworthBoydStability} to this setting, involving a new approach specific to Brauer algebras. In \autoref{Section:Shapiro} we introduce and prove our alternative to Shapiro's Lemma, which includes the proof of~
\autoref{thmaa}. In \autoref{Section: HighConnectivity}, we prove the high connectivity result required in any homological stability proof, using the complex built from the stability categories of~\cite{PatztRepstab} following~\cite{RW}. We bring all our results together in \autoref{Section: MainTheorem}, where we prove \autoref{thma}.

\subsection{Acknowledgments}
The first author would like to thank the Max Planck Institute for Mathematics in Bonn for its support and hospitality.

The third author  would like to thank his advisor Holger Reich for proposing to study diagram algebras as part of his PhD. Further, he would like to thank Reiner Hermann, Steffen Koenig, Jeremy Miller, Steven Sam, and Catharina Stroppel for additional helpful conversations.

\section{Background}\label{Section: Background}

\subsection{Brauer algebras}

In this section we introduce background information on the Brauer algebras, which were introduced by Brauer in 1937~\cite{B}. Other references for Brauer algebras are \cite{Br,Wen}.

\begin{Def}
    Let $R$ be a commutative ring and let $\delta\in R$.
    As described in the introduction, the Brauer algebra $\Br_n(R,\delta) = \Br_n$ is the $R$-algebra with basis given by the diagrams of perfect matchings of~$[-n] \cup [n] = \{-n,\ldots,-1\}\cup\{1,\ldots,n\}$. 
    A perfect matching is a division of this set into~$n$ pairs. 
    These matchings are depicted as diagrams from~$n$ nodes (labelled by~$[-n]$ and drawn on the left) to~$n$ nodes (labelled by~$[n]$ and drawn on the right) where pairs are connected by an arc, also called a \emph{connection}. 
    Connections are defined up to isotopy. 
    Multiplication corresponds to pasting diagrams side by side and replacing each closed loop with a factor of~$\delta$, as illustrated in the introduction. 
    Elements of the Brauer algebra are therefore formal sums of diagrams with coefficients in~$R$. 
\end{Def}

\begin{rem}
    Let~$U_1, \ldots, U_{n-1}$ and $S_1, \ldots, S_{n-1}$ denote the following elements of $\Br_n(R,\delta)$:
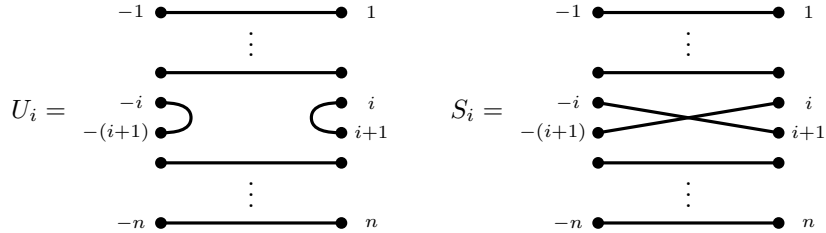
    \[ U_i=
        \begin{tikzpicture}[scale=0.4, baseline=1.8cm]
        \foreach \x in {1, 3,4,5,6,8}
        \foreach \y in {0,6}
        \draw[fill=black, line width=1] (\y,\x) circle [radius=0.15];
        \foreach \x in {1, 3,6,8} 
        \draw[very thick] (0,\x) --(6,\x);
        \foreach \x in {3}
       \draw (\x,2.2) node {$\scriptstyle{\vdots}$} (\x,7.2) node {$\scriptstyle{\vdots}$};
        \draw (-1,1)node {$\scriptstyle{-n}$} 
        (-1.5,4) node {$\scriptstyle{-(i+1)}$}  (-1,5) node {$\scriptstyle{-i}$}   (-1,8) node {$\scriptstyle{-1}$};
        \draw (7,1)node {$\scriptstyle{{n}}$}(7,4)node {$\scriptstyle{{i+1}}$}(7,5)node {$\scriptstyle{{i}}$}(7,8)node {$\scriptstyle{{1}}$};
        \draw[very thick] (0,4) to[out=0,in=-90] (1,4.5) to[out=90,in=0] (0,5);
        \draw[very thick] (6,4) to[out=180,in=-90] (5,4.5) to[out=90,in=180] (6,5);
        \end{tikzpicture}
        \qquad S_i =
            \begin{tikzpicture}[scale=0.4,baseline=1.8cm]
        \foreach \x in {1, 3,4,5,6,8}
        \foreach \y in {0,6}
        \draw[fill=black, line width=1] (\y,\x) circle [radius=0.15];
        \foreach \x in {1, 3,6,8} 
        \draw[very thick] (0,\x) --(6,\x);
        \foreach \x in {3}
       \draw (\x,2.2) node {$\scriptstyle{\vdots}$} (\x,7.2) node {$\scriptstyle{\vdots}$};
            \draw (-1,1)node {$\scriptstyle{-n}$} 
        (-1.5,4) node {$\scriptstyle{-(i+1)}$}  (-1,5) node {$\scriptstyle{-i}$}   (-1,8) node {$\scriptstyle{-1}$};
        \draw (7,1)node {$\scriptstyle{{n}}$}(7,4)node {$\scriptstyle{{i+1}}$}(7,5)node {$\scriptstyle{{i}}$}(7,8)node {$\scriptstyle{{1}}$};
        \draw[very thick] (0,4) -- (6,5);
        \draw[very thick] (0,5)-- (6,4);
        \end{tikzpicture}
    \]
    It is not hard to verify that these elements satisfy the following relations.
    \begin{itemize}
        \item Symmetric group relations:
        \begin{align*}
            S_i^2&=1 && \text{for all }i
            \\
            S_iS_jS_i&=S_jS_iS_j && \text{for }|i-j|=1
            \\
            S_iS_j&=S_jS_i && \text{for }|i-j|\geq 2
        \end{align*}
        \item Temperley-Lieb relations:
        \begin{align*}
            U_i^2& =\delta U_i && \text{for all }i
            \\
            U_iU_jU_i&=U_i && \text{for } |i-j|=1
            \\
            U_iU_j&=U_jU_i && \text{for } |i-j|\geq 2
        \end{align*}
        \item Mixed relations:
        \begin{align*}
            U_iS_i&=S_iU_i=U_i &&\text{for all }i
            \\
            U_iS_jU_i&=U_i &&\text{for } |i-j|=1
            \\
            S_iS_jU_i&=U_jU_i &&\text{for } |i-j|=1
            \\
            U_iS_jS_i&=U_iU_j &&\text{for } |i-j|=1
            \\
            U_iS_j&=S_jU_i&& \text{for } |i-j|\geq2
        \end{align*}
    \end{itemize}
    Indeed, the above is a presentation of $\Br_n(R,\delta)$, though we will not need this fact for the present paper.
\end{rem}

\begin{rem}
The~$S_i$ generators generate the subalgebra $R\fS_n$ of $\Br_n$ obtained from the inclusion $\iota\colon R\fS_n\to\Br_n$ described in the introduction.
Similarly, the~$U_i$ generators generate a subalgebra of the Brauer algebra isomorphic to the Temperley-Lieb algebra~$\TL_n(R,\delta)$.
\end{rem}

We now introduce the modules we will be working with.

\begin{defn}
For any~$n$, we define the \emph{trivial}~$\Br_n$-bimodule~$\t$ to be the module given by the ring~$R$, and upon which the generators~$S_i$ act trivially and the generators~$U_i$ act as zero. Equivalently, any diagram with a left-to-left (and therefore a right-to-right) connection acts as zero, and any other diagram acts as~$1$.
\end{defn}

\begin{defn}
For~$m\leq n$, we can view~$\Br_m$ as the subalgebra of~$\Br_n$ generated by~$S_1,\ldots, S_{m-1},U_1,\ldots,U_{m-1}$, or equivalently as the subalgebra in which the final $n-m$ nodes on the left are connected, by horizontal connections, to the final $n-m$ nodes on the right. 
Then, under the action of this subalgebra,~$\Br_n$ can be viewed as a left $\Br_n$-module and a right~$\Br_m$-bimodule, and we obtain the induced left $\Br_n$-module~$\Br_n\otimes_{\Br_m}\t$.
\end{defn}

The following proposition is taken from~\cite{PatztRepstab}, where the above inclusion of~$\Br_m\hookrightarrow \Br_n$ is taken to be the inclusion into the \emph{final}~$m$ nodes, as opposed to our convention which uses the inclusion into the \emph{first}~$m$ nodes. We adapt the statement of the proposition accordingly.

\begin{prop}[{\cite[Section 2]{PatztRepstab}}]\label{prop:HomBr}
$\Br_n \otimes_{\Br_{m}} \t$ is a free $R$-module and a quotient of~$\Br_n$. 

%Elements are formal sums of monomials which cannot be written to end with~$U_i$ for~$1\leq i \leq m-1$, and for which occurrences of~$S_i$ at the end of monomials can be ignored, for~$1\leq i\leq m-1$. 

In terms of diagrams a basis for this module is the set of diagrams with $n$ nodes on the left and $n-m$ nodes under an \emph{$m$-box} on the right, subject to the following two conditions:
\begin{itemize}
    \item Every node must either be connected to exactly one other node or to the box, and the box must be connected to $m$ nodes.
    \item Any diagram with a connection starting and ending at the box is identified with zero.
\end{itemize} 
\end{prop}

Some (non-)examples of diagrams representing elements of~$\Br_5\otimes_{\Br_3}\t$ are given below. 
\[
\begin{tikzpicture}[x=1.5cm,y=-.5cm,baseline=-1.05cm]

\node[v] (a1) at (0,0) {};
\node[v] (a2) at (0,1) {};
\node[v] (a3) at (0,2) {};
\node[v] (a4) at (0,3) {};
\node[v] (a5) at (0,4) {};

\node[b] (b1) at (1,1) {$3$};
\node[v] (b4) at (1,3) {};
\node[v] (b5) at (1,4) {};

\draw[e] (a1) to[out=0, in=180] (b1);
\draw[e] (a2) to[out=0, in=0]   (a4);
\draw[e] (b1) to[out=180, in=180] (b4);
\draw[e] (a5) to[out=0, in=0] (a3);
\draw[e] (b1) to[out=180,in=180] (b5);

\end{tikzpicture}
\quad,\quad
\begin{tikzpicture}[x=1.5cm,y=-.5cm,baseline=-1.05cm]

\node[v] (a1) at (0,0) {};
\node[v] (a2) at (0,1) {};
\node[v] (a3) at (0,2) {};
\node[v] (a4) at (0,3) {};
\node[v] (a5) at (0,4) {};

\node[b] (b1) at (1,1) {$3$};
\node[v] (b4) at (1,3) {};
\node[v] (b5) at (1,4) {};

\draw[e] (a1) to[out=0, in=180] (b1);
\draw[e] (a2) to[out=0, in=180]   (b1);
\draw[e] (b1) to[out=180, in=0] (a4);
\draw[e] (a5) to[out=0, in=0] (a3);
\draw[e] (b4) to[out=180,in=180] (b5);

\end{tikzpicture}
\quad
\in \Br_5\otimes_{\Br_3}\t,
\qquad
\begin{tikzpicture}[x=1.5cm,y=-.5cm,baseline=-1.05cm]

\node[v] (a1) at (0,0) {};
\node[v] (a2) at (0,1) {};
\node[v] (a3) at (0,2) {};
\node[v] (a4) at (0,3) {};
\node[v] (a5) at (0,4) {};

\node[b] (b1) at (1,1) {$3$};
\node[v] (b4) at (1,3) {};
\node[v] (b5) at (1,4) {};

\draw[e] (a1) to[out=0, in=180] (b1);
\draw[e] (a2) to[out=0, in=0]   (a4);
\draw[e] (b1) to[out=210, in=270] (.5,1);
\draw[e] (b1) to[out=150, in=90] (.5,1);
\draw[e] (a5) to[out=0, in=0] (a3);
\draw[e] (b4) to[out=180,in=180] (b5);

\end{tikzpicture}
\quad
=0\in \Br_5\otimes_{\Br_3}\t\]

\subsection{Tor groups and Shapiro's lemma}

We now recall the definition and functoriality of $\Tor$ groups, and Shapiro's lemma, in order to fix conventions and notation for later.

Let $R$ be a commutative ring.
Let $A$ be an $R$-algebra, $M$ a right $A$-module, and $N$ a left $A$-module.
Choosing a projective (or free if we wish) resolution $P_\ast$ of $M$ over $A$, we may form the tensor product chain complex $P_\ast\otimes_A N$, and its homology is the corresponding $\Tor$ group:
\[
    \Tor_\ast^A(M,N) = H_\ast(P_\ast\otimes N).
\]
Now suppose that $A'$ is a second $R$-algebra, that $M'$ is a right $A$-module, and that $N'$ is a left $A$-module.
Suppose further that there is a ring homomorphism $f\colon A\to A'$, and homomorphisms of $R$-modules $f_R\colon M\to M'$ and $f_L\colon N\to N'$ that are compatible in the sense that
\[
    f_R(ma) = f_R(m)f(a)
    \qquad\text{and}\qquad
    f_L(an) = f(a)f_L(n)
\]
for $m\in M$, $n\in N$ and $a\in A$.
Then there is an \emph{induced map}
\[
    (f_R,f,f_L)_\ast\colon
    \Tor_\ast^A(M,N)
    \longrightarrow
    \Tor_\ast^{A'}(M',N')
\]
or simply 
\[
    f_\ast\colon
    \Tor_\ast^A(M,N)
    \longrightarrow
    \Tor_\ast^{A'}(M',N')
\]
defined as follows.
Choose projective resolutions $P_\ast$ of $M$ over $A$ and $P'_\ast$ of $M'$ over $A'$.
By restricting along $f$ we may regard $P'_\ast$ as a resolution of $M'$ over $A$.
While $P'_\ast$ need not be projective over $A$, it is still acyclic, and we may therefore find a chain map $\tilde f\colon P_\ast\to P'_\ast$ of chain complexes of $A$-modules, inducing $f_R$ on homology.
This induces a chain map $\tilde f\otimes f_L\colon P_\ast\otimes_{A}N\to P'_\ast\otimes_{A'} N'$, and finally determines a map on homology groups, which is $(f_R,f,f_L)_\ast$ above.
Any such $\tilde f$ is determined uniquely up to homotopy, so that the resulting map of $\Tor$-groups is uniquely determined.

\emph{Shapiro's lemma} states that if $G$ is a group and $H$ is a subgroup of $G$, then the natural map
\[
    H_\ast(H;\t)
    =
    \Tor_\ast^{RH}(\t,\t)
    \longrightarrow
    \Tor_\ast^{RG}(\t,RG\otimes_{RH}\t)
    =
    H_\ast(G;RG\otimes_{RH}\t)
\]
induced by the inclusion $RH\hookrightarrow RG$ and the module maps $\mathrm{id}\colon\t\to\t$ and $\t\to RG\otimes_{RH}\t$, $r\mapsto 1\otimes r$, is an isomorphism.
The isomorphism can be realised by choosing a free $RG$-resolution $Q_\ast$ of $\t$, and observing that it is also a free $RH$-resolution of $\t$, so that the induced map above is given by $Q_\ast\otimes_{RH}\t\to Q_\ast\otimes_{RG}(RG\otimes_{RH}\t)$, $q\otimes r\mapsto q\otimes(1\otimes r)$, which is itself an isomorphism of chain complexes.

\section{Inductive resolutions}\label{Section: InductiveResolutions}

In this section, we will adapt the \emph{inductive resolutions} of~\cite{HepworthBoydStability} to the setting of Brauer algebras.

\begin{Def}\label{Def:Jx}
    Suppose that $X$ is a subset of the set $\{1,\ldots,n\}$.
    Define $J_X$ to be the left-ideal in $\Br_n$ that is the $R$-span of all diagrams in which, among the nodes on the right labelled by elements of $X$, at least one pair is connected by an arc.
\end{Def}

Our aim is to prove the following theorem, which will be used in the next section to understand the $\Tor$ groups $\Tor^{\Br_n}_\ast(\t,\Br_n\otimes_{\Br_m}\t)$.

\begin{thm}\label{theorem-inductive}
    Let $X\subseteq\{1,\ldots,n\}$ and suppose that one of the following conditions holds:
    \begin{itemize}
        \item
        $|X|\leq n$ and $\delta$ is invertible in $R$.
        \item
        $|X|<n$.
    \end{itemize}
    Then the groups $\Tor_\ast^{\Br_n}(\t,\Br_n/J_X)$ vanish in positive degrees.
\end{thm}

The proof of this result will occupy the rest of the section.
The material that follows closely parallels \cite[Section~3]{HepworthBoydStability}, with modifications to account for the passage from Temperley-Lieb algebras to Brauer algebras.
The application of \autoref{theorem-inductive} in the next section is entirely new.

Given $a,b\in\{1,\ldots,n\}$, we write $U_{ab}$ for the element of $\Br_n$ represented by the diagram in which each node on the left is joined to the corresponding node on the right, except that nodes $-a,-b$ on the left are joined to one another, and that nodes $a,b$ are joined to one another. 
Thus $U_{i\, i+1}$ is the element that we earlier denoted by $U_i$.

\begin{lem}\label{lemma-Uab}
    \begin{enumerate}[(a)]
        \item\label{lemma-Uab-one}
        An element $\beta \in \Br_n$ has the form $\beta = \alpha\cdot U_{ab}$  for some $\alpha \in\Br_n$ if and only if $\beta$ is a linear combination of diagrams that each have an arc on the right between the nodes labelled $a$ and $b$.
        \item\label{lemma-Uab-two}
        $J_X$ is the left-ideal of $\Br_n$ generated by the elements $U_{ab}$ for $a,b\in X$.
        \item\label{lemma-Uab-three}
        If $a,b\not\in X$, then right-multiplication by $U_{ab}$ preserves the left ideal $J_X$.
    \end{enumerate}
\end{lem}

\begin{proof}
    In \eqref{lemma-Uab-one}, the `only if' part is immediate.
    For the `if' part, consider a diagram $\beta$ in which nodes $a$ and $b$ on the right are connected by an arc.
    Then at least one pair of nodes on the left is connected by an arc.
    Left-multiplying by a permutation (which is invertible), we may assume without loss of generality that nodes $-a$ and $-b$ on the left are also connected by an arc.
    It is then clear that $\beta= \alpha U_{ab}$ where $\alpha$ is obtained from $\beta$ by replacing the arcs from $\pm a$ to $\pm b$ on each side with an arc from $-a$ to $a$ and an arc from $-b$ to $b$.
    
    \eqref{lemma-Uab-two} follows immediately from \eqref{lemma-Uab-one}.
    
    For \eqref{lemma-Uab-three}, $U_{ab}$ commutes with $U_{cd}$ whenever $a,b,c,d$ are all distinct, so that by \eqref{lemma-Uab-two} it commutes with the generators of $J_X$.
\end{proof}

We now introduce two very similar complexes, one for each set of assumptions in \autoref{theorem-inductive}.

\begin{Def}[The complex $C(X,x)$]
    Suppose that $\delta$ is invertible in $R$ and that 
    $X$ is a nonempty subset of $\{1,\ldots, n\}$.
    Let $x\in X$.
    Define a complex $C(X,x)$ as in \autoref{figure-C}.
    \begin{figure}
    \[
        \xymatrix{
            \vdots
            \ar[d]^-{\bigoplus(1-\delta^{-1}U_{xw})}
            \\
            \displaystyle\bigoplus_{w\in X-\{x\}}
            \Br_n/J_{X-\{x,w\}}
            \ar[d]^-{\bigoplus \delta^{-1}U_{xw}}
            &
            3
            \\
            \displaystyle\bigoplus_{w\in X-\{x\}}
            \Br_n/J_{X-\{x,w\}}
            \ar[d]^-{\bigoplus (1-\delta^{-1}U_{xw})}
            &
            2
            \\
            \displaystyle\bigoplus_{w\in X-\{x\}}
            \Br_n/J_{X-\{x,w\}}
            \ar[d]^-{\bigoplus \delta^{-1}U_{xw}}
            &
            1
            \\
            \Br_n/J_{X-\{x\}}
            \ar[d]^1
            &
            0
            \\
            \Br_n/J_X
            &
            -1
        }
    \]
    \caption{The complex $C(X,x)$.}
    \label{figure-C}
    \end{figure}
    Thus $C(X,x)$ is given in degree $-1$ by $\Br_n/J_X$, in degree $0$ by $\Br_n/J_{X-\{x\}}$, and in all higher degrees by the direct sum of $\Br_n/J_{X-\{x,w\}}$ for $w\in X-\{x\}$.
    Each differential is given on each direct summand by a combination of right-multiplication by elements of $\Br_n$ (these are the elements indicated on the arrows of the diagram) and a map that extends the ideal by which we are quotienting (these are clear from the notation).
\end{Def}

\begin{Def}[The complex $D(X,x,y)$]
    Let $X$ be a proper nonempty subset of $\{1,\ldots, n\}$.
    Let $x\in X$ and let $y\in\{1,\ldots,n\}-X$.
    Define a complex $D(X,x,y)$ as in \autoref{figure-D}.
    \begin{figure}
    \[
        \xymatrix{
            \vdots
            \ar[d]^-{\bigoplus(1-U_{xw}U_{xy})}
            \\
            \displaystyle\bigoplus_{w\in X-\{x\}}
            \Br_n/J_{X-\{x,w\}}
            \ar[d]^-{\bigoplus U_{xw}U_{xy}}
            &
            3
            \\
            \displaystyle\bigoplus_{w\in X-\{x\}}
            \Br_n/J_{X-\{x,w\}}
            \ar[d]^-{\bigoplus (1-U_{xw}U_{xy})}
            &
            2
            \\
            \displaystyle\bigoplus_{w\in X-\{x\}}
            \Br_n/J_{X-\{x,w\}}
            \ar[d]^-{\bigoplus U_{xw}}
            &
            1
            \\
            \Br_n/J_{X-\{x\}}
            \ar[d]^1
            &
            0
            \\
            \Br_n/J_X
            &
            -1
        }
    \]
    \caption{The complex $D(X,x,y)$.}
    \label{figure-D}
    \end{figure}
    Thus $D(X,x,y)$ is given in degree $-1$ by $\Br_n/J_X$, in degree $0$ by $\Br_n/J_{X-\{x\}}$, and in all higher degrees by the direct sum of the $\Br_n/J_{X-\{x,w\}}$ for $w\in X-\{x\}$.
    And as in the previous definition, each differential is given on each direct summand by a combination of right-multiplication by an element of $\Br_n$ and a map that extends the ideal by which we are quotienting.
\end{Def}

We will see shortly that $C(X,x)$ and $D(X,x,y)$ are acyclic, i.e.~they are resolutions of their degree $(-1)$ part.
We call them \emph{inductive resolutions} because they are resolving a module $\Br_n/J_X$ in degree $(-1)$, while in all remaining degrees they are built from modules of the form $\Br_n/J_Y$ for $|Y|<|X|$.
This will facilitate the inductive proof of~\autoref{theorem-inductive} that appears at the end of the section.

\begin{lem}
    The differentials of $C(X,x)$ and $D(X,x,y)$ are well defined, and consecutive differentials compose to give $0$.
\end{lem}

\begin{proof}
    The differentials are well-defined because we are always right-multiplying by an element that commutes with the ideal appearing in the domain. See \autoref{lemma-Uab}\eqref{lemma-Uab-three}.
    
    In both cases the composite of differentials from degree $1$ to $-1$ vanishes because the elements $U_{xw}$ all lie in $J_X$. 
    
    In $C(X,x)$ the remaining composites all vanish because each $\delta^{-1}U_{xw}$ is an idempotent:
    \[
        (\delta^{-1}U_{xw})^2
        =\delta^{-2}U_{xw}^2
        =\delta^{-2}\delta U_{xw}
        =\delta^{-1}U_{xw}
    \]
    
    In $D(X,x,y)$ the composite of differentials ending in degree $0$ vanishes because of the computation:
    \[
        (1-U_{xw}U_{xy})U_{xw}
        =
        U_{xw} - U_{xw}U_{xy}U_{xw}
        =
        U_{xw}-U_{xw}
        =0.
    \]
    Here the identity $U_{xw}U_{xy}U_{xw} = U_{xw}$ is easily verified, and generalises the familiar identity $U_iU_{i\pm 1}U_i=U_i$ in the Temperley-Lieb algebra.
    The remaining composites of consecutive differentials in $D(X,x,y)$ all vanish because the element $U_{xw}U_{xy}$ is idempotent:
    \[
        (U_{xw}U_{xy})(U_{xw}U_{xy})
        =
        U_{xw}(U_{xy}U_{xw}U_{xy})
        =
        U_{xw}U_{xy}.\qedhere
    \]
\end{proof}

\begin{lem}\label{lemma-D-acyclic}
    The complexes $C(X,x)$ and $D(X,x,y)$ are acyclic.
\end{lem}

\begin{proof}
    We begin with the proof for $C(X,x)$.

    In degree $-1$ the claim is clear since $\Br_n/J_{X-\{x\}}\to\Br_n/J_X$ is surjective.
    
    In degree $0$, we must show that any element of $J_X$ is equal, modulo $J_{X-\{x\}}$, to a sum of left-multiples of the elements $\delta^{-1}U_{xw}$ for $w\in X-\{x\}$.
    This is an immediate consequence of \autoref{lemma-Uab}\eqref{lemma-Uab-one}.

    Before proceeding to higher degrees, we consider the boundary map of $C(X,x)$ from degree $1$ to degree $0$:
    \[
    \xymatrix{
        \displaystyle\bigoplus_{w\in X-\{x\}}
        \Br_n/J_{X-\{x,w\}}
        \ar[d]^-{\bigoplus \delta^{-1} U_{xw}}
        &
        1
        \\
        \Br_n/J_{X-\{x\}}
        &
        0
    }
    \]
    According to \autoref{lemma-Uab}\eqref{lemma-Uab-one},
    the image of the summand $\Br_n/J_{X-\{x,w\}}$ is spanned by linear combinations of diagrams in which the nodes labelled $x$ and $w$ on the right are connected by an arc. 
    In any diagram, the node $x$ on the right is connected to one other node, and no more, so that the images of the summands  $\Br_n/J_{X-\{x,w\}}$ for different choices of $w$ form a direct sum in $\Br_n/J_{X-\{x\}}$.
    Therefore, in order to prove acyclicity of our complex in degrees $1$ and above, it will be enough to fix $w\in X-\{x\}$ and prove the complex $C(X,x,w)$ shown in \autoref{figure-C-summand} has vanishing homology in degree $1$ and above.
    \begin{figure}
    \[
        \xymatrix{
            \vdots
            \ar[d]^-{ \delta^{-1}U_{xy}}
            &
            \\
            \Br_n/J_{X-\{x,w\}}
            \ar[d]^-{ (1-\delta^{-1}U_{xw})}
            &
            2
            \\
            \Br_n/J_{X-\{x,w\}}
            \ar[d]^-{ \delta^{-1}U_{xw}}
            &
            1
            \\
            \Br_n/J_{X-\{x\}}
            &
            0
        }
    \]
    \caption{The complex $C(X,x,w)$}\label{figure-C-summand}
    \end{figure}
    
    We now prove that the homology of $C(X,x,w)$ vanishes in degrees $1$ and above.
    In degree $1$, suppose that we have an element $\alpha\in\Br_n$ that represents a cycle in degree $1$, so that $\alpha \delta^{-1} U_{xw}\in J_{X-\{x\}}$.  
    Consider a diagram appearing as a summand of $\alpha\delta^{-1} U_{xw}$ with nonzero coefficient.
    Since there is a factor of $U_{xw}$ on the right, the diagram must have an arc between the nodes on the right labelled $x$ and $w$.
    And since $\alpha U_{xw}$ lies in $J_{X-\{x\}}$, there are $a,b\in X-\{x\}$ such that the nodes labelled $a,b$ on the right of the diagram are joined by an arc. Comparing the last two sentences, we see that in fact $a,b \in X-\{x,w\}$, since otherwise the diagram would have two distinct arcs ending at $w$, and consequently the diagram lies in $J_{X-\{x,w\}}$.
    Consequently, $\alpha \delta^{-1} U_{xw}$ itself lies in $J_{X-\{x,w\}}$.
    This shows that $\alpha = \alpha(1-\delta^{-1}U_{xw})$ in $\Br_n/J_{X-\{x,w\}}$, so that $\alpha$ lies in the image of the differential.
    In degrees $2$ and higher exactness is immediate, thanks to the fact that $\delta^{-1}U_{xw}$ is idempotent.
    
    Now we move on to the proof for $D(X,x,y)$.
    All steps of the proof are closely analogous to the proof above for $C(X,x)$, with the complex $D(X,x,y,w)$ shown in \autoref{figure-D-summand} now playing the role of $C(X,x,w)$.
    \begin{figure}
    \[
        \xymatrix{
            \vdots
            \ar[d]^-{(1-U_{xw}U_{xy})}
            \\
            \Br_n/J_{X-\{x,w\}}
            \ar[d]^-{ U_{xw}U_{xy}}
            &
            3
            \\
            \Br_n/J_{X-\{x,w\}}
            \ar[d]^-{ (1-U_{xw}U_{xy})}
            &
            2
            \\
            \Br_n/J_{X-\{x,w\}}
            \ar[d]^-{ U_{xw}}
            &
            1
            \\
            \Br_n/J_{X-\{x\}}
            &
            0
        }
    \]
    \caption{The complex $D(X,x,y,w)$}\label{figure-D-summand}
    \end{figure}
    The only point of departure is the proof that $D(X,x,y,w)$ is exact in degree $1$, which we prove now.
    In $D(X,x,y,w)$, suppose that we have an element $\alpha\in\Br_n$ that represents a cycle in degree $1$, so that $\alpha U_{xw}\in J_{X-\{x\}}$.  
    Just as for $C(X,x,w)$, it follows that  $\alpha U_{xw}$ itself lies in $J_{X-\{x,w\}}$.
    By \autoref{lemma-Uab}\eqref{lemma-Uab-three}, right-multiplication by $U_{xy}$ preserves the left-ideal $J_{X-\{x,w\}}$, so that $\alpha U_{xw}U_{xy}$ lies in $J_{X-\{x,w\}}$.
    This shows that $\alpha = \alpha(1-U_{xw}U_{xy})$ in $\Br_n/J_{X-\{x,w\}}$, so that $\alpha$ lies in the image of the differential.
\end{proof}

\begin{lem}\label{lemma-D-tensor-acyclic}
    The complexes $\t\otimes_{\Br_n}C(X,x)$ and $\t\otimes_{\Br_n}D(X,x,y)$ are acyclic.
\end{lem}

\begin{proof}
    If $A$ is any subset of $\{1,\ldots,n\}$, then all elements of the ideal $J_A$ act as zero on $\t$.
    It follows that $\t\otimes_{\Br_n}\Br_n/J_A$ is isomorphic to $R$.
    
    Under this isomorphism, the differentials on each summand of $(\t\otimes_{\Br_n}C(X,x))_p$ and $(\t\otimes_{\Br_n}D(X,x,y))_p$ are given by~$0$ for~$p$ odd and~$\id$ for~$p$ even (they act as the indicated element in \autoref{figure-C} and~\autoref{figure-D} would act on~$\t$).
    Using this description, one sees immediately that the complexes are acyclic.
\end{proof}

\begin{proof}[Proof of \autoref{theorem-inductive}]
    We give the proof for $C(X,x)$; the proof for $D(X,x,y)$ is identical.

    The proof is by induction on the cardinality of $X$.
    The cases where $X$ has cardinality $0$ or $1$ are immediate because then $J_X=0$ and $\Br_n/J_X = \Br_n$ is free.
    Assume that $X$ has cardinality $2$ or more, choose $x\in X$, and choose $y\not\in X$.
    
    Let $C_{\geq 0}(X,x)$ denote the part of  $C(X,x)$ in non-negative degrees.
    By \autoref{lemma-D-acyclic}, $C(X,x)$ is acyclic, and so $C_{\geq 0}(X,x)$ is a resolution of $\Br_n/J_X$. In each degree, $C_{\geq 0}(X,x)$ is a direct sum of modules~$\Br_n/J_Y$ for~$Y \subsetneq X$. 
    By the inductive hypothesis the $\Tor$ groups of these modules vanish in positive degrees, and it follows that $C_{\geq 0}(X,x)$ is a \emph{flat} resolution of $\Br_n/J_X$. 
    Therefore $\Tor^{\Br_n}_*(\t,\Br_n/J_X)$ is computed by $\t \otimes_{\Br_n} C_{\geq 0}(X,x)$, which by \autoref{lemma-D-tensor-acyclic} is zero in positive degrees. \qedhere
\end{proof}

\section{Replacing Shapiro's Lemma}\label{Section:Shapiro}

In this section we discuss the failure of Shapiro's lemma for Brauer algebras, we provide a replacement for it, and we prove \autoref{thmaa}.

Recall that Shapiro's lemma states that if $G$ is a group and $H$ is a subgroup of $G$, then the natural map
\[
    H_\ast(H;\t)
    =
    \Tor_\ast^{RH}(\t,\t)
    \longrightarrow
    \Tor_\ast^{RG}(\t,RG\otimes_{RH}\t)
\]
is an isomorphism.
The case of interest to us is the following, where $m\leq n$:
\[
    \sigma\colon
    H_\ast(\fS_m;\t)
    =
    \Tor_\ast^{R\fS_m}(\t,\t)
    \xrightarrow{\ \cong\ }
    \Tor_\ast^{R\fS_n}(\t,R\fS_n\otimes_{R\fS_m}\t)
\]
In this specific case we will always denote the isomorphism by $\sigma$.

If it were possible, we would also be interested in a version of Shapiro's lemma for the Brauer algebra, identifying $\Tor_\ast^{\Br_m}(\t,\t)$ with $\Tor_\ast^{\Br_n}(\t,\Br_n\otimes_{\Br_m}\t)$.
The reason we would like this is that the highly-connected chain complex that we will exploit later is built out of precisely the modules $\Br_n\otimes_{\Br_m}\t$, as is familiar from proofs of homological stability for families of groups.
However no such version of Shapiro's lemma for Brauer algebras is possible, because $\Br_n$ is not necessarily flat as a right $\Br_m$-module.  
A concrete instance of this is given by the fact that  $\Tor^{\Br_2}_1(\Br_3,\t)\cong (R/\delta R)^{\oplus 3}$,
which shows that $\Br_3$ is not flat over $\Br_2$ when $\delta$ is not invertible.
This computation of $\Tor^{\Br_2}_1(\Br_3,\t)$ is obtained by considering the short exact sequence of right $\Br_2$-modules
\[
    0\to J\to\Br_3\to\Br_3/J\to0,
\]
where $J=J_{\{1,2,3\}}$ is the ideal spanned by diagrams with a left-to-left connection;
one can identify $J$ and $\Br_3/J$, and then compute their $\Tor^{\Br_2}_1(-,\t)$ using the method of~\cite[3.1.3]{Wei}.

% \begin{prop}
%     $\Tor^{\Br_2}_1(\Br_3,\t)\cong (R/\delta R)^{\oplus 3}$.
% \end{prop}

% \begin{proof}
%     Note that $\Br_2$ has $R$-basis consisting of the elements $1$, $U_1$ and $S_1$, with multiplication given by $U_1^2=\delta U_1$, $S_1^2=1$, $U_1S_1 = S_1U_1 = U_1$.
    
%     Let $J_3\subseteq \Br_3$ denote the span of all diagrams which have left-to-left connections.
%     Consider the short exact sequence of right $\Br_2$-modules:
%     \[
%         0\to J_3\to\Br_3\to\Br_3/J_3\to0
%     \]
%     We have $J_3\cong\Br_2^{\oplus 3}$, with generators the three diagrams in which nodes $2$ and $3$ are connected on the right.
%     And we have $\Br_3/J_3\cong\bbS^{\oplus 3}$, where $\bbS$ is a copy of $R\oplus R$ on which $S_1$ acts by transposing the summands, and on which $U_1$ acts as $0$.
%     The three generators of $\Br_3/J_3$ are the diagrams corresponding to the permutations $\mathrm{id}$, $(1\,2\,3)$ and $(1\,3\,2)$
    
%     From the above paragraph it follows quickly that $\Tor_1^{\Br_2}(\Br_3,\t)\cong\Tor_1^{\Br_2}(\t,\bbS)^{\oplus 3}$.
%     Now $\t=\Br_2/I$ and $\bbS=\Br_2/J$, where $I$ is the right-ideal generated by $U_1$ and $1-S_1$, and $J$ is the right-ideal generated by $U_1$, so that
%     $\Tor_1^{\Br_2}(\t,\bbS)\cong I\cap J/I\cdot J$.
%     We see that $I\cap J$ is the $R$-span of $U_1$, while $I\cdot J$ is the $R$-span of $\delta U_1$.
%     Thus $\Tor_1^{\Br_2}(\t,\bbS)\cong R/\delta R$, and the result follows.
% \end{proof}

This absence of Shapiro's lemma is what prevents us from presenting a `traditional' homological stability proof for Brauer algebras.
However, we are able instead to use the inductive resolutions of the previous section to prove a `replacement' for Shapiro's lemma, that will turn out to be just as useful, if not more.

To phrase the main result of this section, recall the inclusion and projection maps
\[
    R\fS_m\xrightarrow{\iota} \Br_m\xrightarrow{\pi} R\fS_m.
\]
These are compatible with the inclusions $\Br_{m}\to\Br_n$ and $R\fS_{m}\to R\fS_n$, and also respect the actions on the trivial module.  
They therefore induce the following maps of $\Tor$-groups.
\[
    \Tor_\ast^{R\fS_n}(\t,R\fS_n\otimes_{R\fS_m}\t)
    \xrightarrow{\iota_\ast}
    \Tor_\ast^{\Br_n}(\t,\Br_n\otimes_{\Br_m}\t)
    \xrightarrow{\pi_\ast}
    \Tor_\ast^{R\fS_n}(\t,R\fS_n\otimes_{R\fS_m}\t)
\]

\begin{thm}\label{prop-tor-quotient-Homology-Sm}
Let $n\geqslant m\geqslant 0$. 
Suppose that $\delta$ is invertible in $R$, or that $m<n$.
Then the maps
\[
    \iota_\ast\colon
    \Tor_\ast^{R\fS_n}(\t,R\fS_n\otimes_{R\fS_m}\t)
    \longrightarrow
    \Tor_\ast^{\Br_n}(\t,\Br_n\otimes_{\Br_m}\t)
\]
and
\[
    \pi_\ast\colon
    \Tor_\ast^{\Br_n}(\t,\Br_n\otimes_{\Br_m}\t)
    \longrightarrow
    \Tor_\ast^{R\fS_n}(\t,R\fS_n\otimes_{R\fS_m}\t)
\]
are mutually inverse isomorphisms.
\end{thm}

\begin{proof}[Proof of \autoref{thmaa}]
    Taking $\delta$ invertible and $m=n$, the result follows immediately using the identifications    $R\fS_n\otimes_{R\fS_m}\t\cong\t$ and  $\Br_n\otimes_{\Br_m}\t\cong \t$.
\end{proof}

\autoref{prop-tor-quotient-Homology-Sm} will be proved after some preparatory definitions and lemmas.

\begin{defn}\label{def-Jm}
For~$m\leq n$, let~$J_m$ be the left-ideal~$J_X$ defined in \autoref{Def:Jx}  for~$X=\{1,\ldots ,m\}$.
Thus $J_m$ is the left-ideal spanned by all diagrams that have at least one arc among the nodes labelled $1,\ldots,m$ on the right.
\end{defn}

Observe that $\Br_n$ is a right $R\fS_m$-module, and that this module structure preserves $J_m$, so that $\Br_n/J_m$ becomes a right $R\fS_m$-module.

\begin{lem}\label{lem-quotient-by-Jm-free}
For~$m\leq n$, $\Br_n/J_m$ is a free~$R\fS_m$-module.
\begin{proof}
$\Br_n/J_m$ has basis consisting of the diagrams that have no arc between any two nodes in $\{1,\ldots,m\}$.
In fact, $\fS_m$ acts freely on this basis. This follows from the observation that multiplying such a diagram with a permutation in~$\fS_m$ results again in a diagram with no arc between any two nodes in $\{1,\dots,m\}$ and the stabilizer of any such diagram is only the trivial permutation.
\end{proof}
\end{lem}

\begin{lem}\label{lem-quotient-by-Jm-tensor-product}
    For~$m\leq n$, there is an isomorphism of left $\Br_n$-modules
    \[
    \Br_n/J_m\otimes_{R\fS_m}\t \cong \Br_n\otimes_{\Br_m}\t 
    \]
    under which $(b+J_m)\otimes r\in\Br_n/J_m\otimes_{R\fS_m}\t$ corresponds to $b\otimes r\in \Br_n\otimes_{\Br_m}\t$.
\end{lem}

\begin{proof}
It is sufficient to check that the correspondence described in the statement is a well-defined map in each direction.
This holds because elements of $R\fS_m$ act as the identity on $\t$ while elements of $J_m$ act as $0$, and because $\Br_m$ is spanned by $R\fS_m$ and $J_m$.
\end{proof}

Now recall from \autoref{theorem-inductive} that, under  the hypotheses of \autoref{prop-tor-quotient-Homology-Sm},
\[
\Tor^{\Br_n}_\ast (\t, \Br_n/J_m)=\begin{cases} \t &\mbox{if } \ast= 0 \\
0 & \mbox{if } \ast>0 \end{cases}.
\]

\begin{proof}[Proof of \autoref{prop-tor-quotient-Homology-Sm}]
Since $\pi_\ast\circ\iota_\ast$ is the identity map, it is sufficient to show that $\iota_\ast$ is an isomorphism.
To do this, we will directly construct an isomorphism
\[
    \Theta
    \colon
    \Tor_*^{\Br_n}(\t,\Br_n\otimes_{\Br_m}\t)
    \xrightarrow{\cong} 
    H_\ast(\fS_m;\t).
\]
and show that the composite 
\[
    H_\ast(\fS_m;\t)
    \xrightarrow[\cong]{\sigma}
    \Tor_\ast^{R\fS_n}(\t,R\fS_n\otimes_{R\fS_m}\t)
    \xrightarrow{\iota_\ast}
    \Tor_\ast^{\Br_n}(\t,\Br_n\otimes_{\Br_m}\t)
    \xrightarrow[\cong]{\Theta}
    H_\ast(\fS_m;\t)
\]
is the identity map.

Let us construct the map $\Theta$.
Let~$P_\ast$ be a free~$\Br_n$-resolution of $\t$. 
Then the domain of $\Theta$ is the homology of the chain complex $P_\ast\otimes_{\Br_n}\Br_n\otimes_{\Br_m}\t$, and \autoref{lem-quotient-by-Jm-tensor-product} gives us an isomorphism
\[
    \theta
    \colon 
    P_\ast\otimes_{\Br_n}\Br_n\otimes_{\Br_m}\t 
    \xrightarrow{\ \cong \ }
    P_\ast\otimes_{\Br_n}\Br_n/J_m\otimes_{R\fS_m}\t. 
\]
Since $P_\ast\otimes_{\Br_n}\Br_n/J_m$ computes $\Tor_*^{\Br_n}(\t,\Br_n/J_m)$, \autoref{theorem-inductive} shows that $P_\ast\otimes_{\Br_n}\Br_n/J_m$ is a $\fS_m$-resolution of $\t$.
And since $P_\ast$ is a free $\Br_n$-resolution and $\Br_n/J_m$ is a free $\fS_m$-module, $P_\ast\otimes_{\Br_n}\Br_n/J_m$ is in fact a free $\fS_m$-resolution of $\t$.
Therefore $P_\ast\otimes_{\Br_n}\Br_n/J_m\otimes_{R\fS_m}\t$ is a chain complex whose homology is~$H_\ast(\fS_m;\t)$.
Let $\Theta$ be the map induced by $\theta$.

We now show that the composite $\Theta\circ\iota_\ast\circ\sigma$ is the identity map.  
If we choose a projective  $R\fS_n$-resolution $Q_\ast$ of $\t$, then $\sigma$ is given on the chain level by 
\[
    Q_\ast\otimes_{R\fS_m}\t
    \to 
    Q_\ast\otimes_{R\fS_n}R\fS_n\otimes_{R\fS_m}\t,
    \qquad
    q\otimes r\mapsto q\otimes 1\otimes r.
\]
If we choose a chain map $\tilde\iota\colon Q_\ast\to P_\ast$ that lies over the identity map on $\t$ and respects the inclusion $\iota\colon R\fS_n\to\Br_n$, then $\iota_\ast$ is represented by the chain map 
\[
    Q_\ast\otimes_{R\fS_n}R\fS_n\otimes_{R\fS_m}\t 
    \to 
    P_\ast\otimes_{\Br_n}\Br_n\otimes_{\Br_m}\t,
    \qquad
    q\otimes x\otimes r
    \mapsto
    \tilde\iota(q)\otimes\iota(x)\otimes r.
\]
Finally, we described $\Theta$ explicitly on the chain level in the last paragraph. 
Now we can verify that $\Theta\circ\iota_\ast\circ\sigma$ is given on the chain level by
\[
    Q_\ast\otimes_{R\fS_m}\t
    \to
    P_\ast\otimes_{\Br_n}\Br_n/J_m\otimes_{R\fS_m}\t,
    \qquad
    q\otimes r\mapsto \tilde\iota(q)\otimes(1+J_m)\otimes r.
\]
This map is obtained by applying $-\otimes_{R\fS_m}\t$ to the map
\[
    Q_\ast
    \to
    P_\ast\otimes_{\Br_n}\Br_n/J_m
    \qquad
    q\mapsto \tilde\iota(q)\otimes(1+J_m).
\]
But this is a map of projective $R\fS_m$-resolutions of $\t$, lying above the identity map on $\t$, and respecting the module structure. 
It is therefore a chain homotopy equivalence, and the same therefore holds for our chain-level representative of $\Theta\circ\iota_\ast\circ\sigma$.
This completes the proof.
\end{proof}

\section{High connectivity}\label{Section: HighConnectivity}

In this section, we introduce a chain complex built from the induced modules $\Br_n\otimes_{\Br_m}\t$ of \autoref{prop:HomBr}. This chain complex is analogous to the chain complexes used in Randal-Williams--Wahl \cite{RW} when considering the stability category $\CBr$ from \cite{PatztRepstab}. Furthermore, as generically required for a homological stability proof, we show that this is highly acyclic.

\subsection{The chain complex}

\begin{Def}\label{def:chain complex}
For~$n$ a non-negative integer, we define the chain complex $C_n=(C_n)_\ast$ of~$\Br_n$-modules as follows. The degree~$p$ part $(C_n)_p$ is non-zero in degrees $-1\leq p \leq n-1$, where it is given by
\[
(C_n)_p=\Br_n\otimes_{\Br_{n-(p+1)}}\t.
\]
So in degree~$-1$ it follows that~$(C_n)_{-1}=\Br_n\otimes_{\Br_{n}}\t\cong\t$.
For~$0\leq p \leq n-1$ the degree~$p$ differential~$\partial^p$ is given by the alternating sum
\[ \partial^p = \sum_{i=0}^{p} (-1)^i d^p_i \colon (C_n)_p \longrightarrow (C_n)_{p-1}.\]
Where, algebraically, the map~$d^p_i$ for~$0\leq i\leq p$ is given by
\begin{eqnarray*}
d^p_i \colon \Br_n \tens[\Br_{n-(p+1)}] \t &\longrightarrow& \Br_n \tens[\Br_{n-p}] \t\\
x\otimes r &\mapsto& (x\cdot S_{n-p+i-1}\cdots S_{n-p})\otimes r.
\end{eqnarray*}

In terms of diagrams, elements in degree~$p$ can be described as in \autoref{prop:HomBr}, i.e.~as diagrams with an $(n-(p+1))$-box at the top right. The map~$d^p_i$ first connects the $(n-p+i)$-th node of the right hand side of the diagram with the~$(n-p)$-th node (the top node under the box) and then extends the box over this node, corresponding to the $(n-(p+1))$-box growing by one node to form an $(n-p)$-box.
In other words, if the nodes below the $(n-(p+1))$-box are labelled $0,\ldots,p$ from top to bottom, then $d_i^p$ lifts up node $i$ and plugs it into the box.
\end{Def}

Examples of the action of the maps~$d^p_i$ when~$n=5$ and~$p=2$ are shown below.

\begin{gather*}
d^2_2\left(
\begin{tikzpicture}[scale=0.9, x=1.5cm,y=-.5cm,baseline=-1.05cm]
\node[v] (a1) at (0,0) {};
\node[v] (a2) at (0,1) {};
\node[v] (a3) at (0,2) {};
\node[v] (a4) at (0,3) {};
\node[v] (a5) at (0,4) {};
\node[b] (b1) at (1,0) {$2$};
\node[v] (b3) at (1,2) {};
\node[v] (b4) at (1,3) {};
\node[v] (b5) at (1,4) {};
\draw[e] (a2) to[out=0, in=180] (b1);
\draw[e] (a1) to[out=0, in=0]   (a4);
\draw[e] (b5) to[out=180, in=180] (b4);
\draw[e] (a3) to[out=0, in=180] (b3);
\draw[e] (a5) to[out=0,in=180] (b1);
\end{tikzpicture}
\right)\,=\quad
\begin{tikzpicture}[scale=0.9, x=1.5cm,y=-.5cm,baseline=-1.05cm]
\node[v] (a1) at (0,0) {};
\node[v] (a2) at (0,1) {};
\node[v] (a3) at (0,2) {};
\node[v] (a4) at (0,3) {};
\node[v] (a5) at (0,4) {};
\node[b] (b1) at (1,0) {$3$};
\node[v] (b3) at (1,3) {};
\node[v] (b4) at (1,4) {};
\draw[e] (a2) to[out=0, in=180] (b1);
\draw[e] (a1) to[out=0, in=0]   (a4);
\draw[e] (b1) to[out=180, in=180] (b4);
\draw[e] (a3) to[out=0, in=180] (b3);
\draw[e] (a5) to[out=0,in=180] (b1);
\end{tikzpicture}
\quad
,
\qquad
d^2_1\left(
\begin{tikzpicture}[scale=0.9, x=1.5cm,y=-.5cm,baseline=-1.05cm]
\node[v] (a1) at (0,0) {};
\node[v] (a2) at (0,1) {};
\node[v] (a3) at (0,2) {};
\node[v] (a4) at (0,3) {};
\node[v] (a5) at (0,4) {};
\node[b] (b1) at (1,0) {$2$};
\node[v] (b3) at (1,2) {};
\node[v] (b4) at (1,3) {};
\node[v] (b5) at (1,4) {};
\draw[e] (a2) to[out=0, in=180] (b1);
\draw[e] (a1) to[out=0, in=0]   (a4);
\draw[e] (b5) to[out=180, in=180] (b4);
\draw[e] (a3) to[out=0, in=180] (b3);
\draw[e] (a5) to[out=0,in=180] (b1);
\end{tikzpicture}
\right)\,=\quad
\begin{tikzpicture}[scale=0.9, x=1.5cm,y=-.5cm,baseline=-1.05cm]
\node[v] (a1) at (0,0) {};
\node[v] (a2) at (0,1) {};
\node[v] (a3) at (0,2) {};
\node[v] (a4) at (0,3) {};
\node[v] (a5) at (0,4) {};
\node[b] (b1) at (1,0) {$3$};
\node[v] (b3) at (1,3) {};
\node[v] (b5) at (1,4) {};
\draw[e] (a2) to[out=0, in=180] (b1);
\draw[e] (a1) to[out=0, in=0]   (a4);
\draw[e] (b5) to[out=180, in=180] (b1);
\draw[e] (a3) to[out=0, in=180] (b3);
\draw[e] (a5) to[out=0,in=180] (b1);
\end{tikzpicture}
\\
\\
d^2_0\left(
\begin{tikzpicture}[scale=0.9, x=1.5cm,y=-.5cm,baseline=-1.05cm]
\node[v] (a1) at (0,0) {};
\node[v] (a2) at (0,1) {};
\node[v] (a3) at (0,2) {};
\node[v] (a4) at (0,3) {};
\node[v] (a5) at (0,4) {};
\node[b] (b1) at (1,0) {$2$};
\node[v] (b3) at (1,2) {};
\node[v] (b4) at (1,3) {};
\node[v] (b5) at (1,4) {};
\draw[e] (a2) to[out=0, in=180] (b1);
\draw[e] (a1) to[out=0, in=0]   (a4);
\draw[e] (b5) to[out=180, in=180] (b4);
\draw[e] (a3) to[out=0, in=180] (b3);
\draw[e] (a5) to[out=0,in=180] (b1);
\end{tikzpicture}
\right)\,=\quad
\begin{tikzpicture}[scale=0.9, x=1.5cm,y=-.5cm,baseline=-1.05cm]
\node[v] (a1) at (0,0) {};
\node[v] (a2) at (0,1) {};
\node[v] (a3) at (0,2) {};
\node[v] (a4) at (0,3) {};
\node[v] (a5) at (0,4) {};
\node[b] (b1) at (1,0) {$3$};
\node[v] (b4) at (1,3) {};
\node[v] (b5) at (1,4) {};
\draw[e] (a2) to[out=0, in=180] (b1);
\draw[e] (a1) to[out=0, in=0]   (a4);
\draw[e] (b5) to[out=180, in=180] (b4);
\draw[e] (a3) to[out=0, in=180] (b1);
\draw[e] (a5) to[out=0,in=180] (b1);
\end{tikzpicture}
\quad
,
\qquad
d^2_0\left(
\begin{tikzpicture}[scale=0.9, x=1.5cm,y=-.5cm,baseline=-1.05cm]
\node[v] (a1) at (0,0) {};
\node[v] (a2) at (0,1) {};
\node[v] (a3) at (0,2) {};
\node[v] (a4) at (0,3) {};
\node[v] (a5) at (0,4) {};
\node[b] (b1) at (1,0) {$2$};
\node[v] (b3) at (1,2) {};
\node[v] (b4) at (1,3) {};
\node[v] (b5) at (1,4) {};
\draw[e] (a2) to[out=0, in=180] (b1);
\draw[e] (a1) to[out=0, in=0]   (a4);
\draw[e] (b1) to[out=180, in=180] (b3);
\draw[e] (a3) to[out=0, in=180] (b4);
\draw[e] (a5) to[out=0,in=180] (b5);
\end{tikzpicture}
\right)\, = \quad
\begin{tikzpicture}[scale=0.9, x=1.5cm,y=-.5cm,baseline=-1.05cm]
\node[v] (a1) at (0,0) {};
\node[v] (a2) at (0,1) {};
\node[v] (a3) at (0,2) {};
\node[v] (a4) at (0,3) {};
\node[v] (a5) at (0,4) {};
\node[b] (b1) at (1,0) {$3$};
\node[v] (b4) at (1,3) {};
\node[v] (b5) at (1,4) {};
\draw[e] (a2) to[out=0, in=180] (b1);
\draw[e] (a1) to[out=0, in=0]   (a4);
\draw[e] (b1) to[out=210, in=270] (.5,0);
\draw[e] (b1) to[out=150, in=90] (.5,0);
\draw[e] (a3) to[out=0, in=180] (b4);
\draw[e] (a5) to[out=0,in=180] (b5);
\end{tikzpicture}= 0 \quad 
\end{gather*}

\begin{rem}
The complex~$C_n$ is analogous to the semisimplicial set defined in Randal-Williams--Wahl \cite{RW}, using the symmetric monoidal category~$\mathcal{C}_{\Br}$ introduced by the third author in~\cite{PatztRepstab}. However we do not exploit this analogy, as our numbering system opposes that of~\cite{PatztRepstab}.
\end{rem}

\begin{lem}\label{lemma-complex}
    $(C_n)_\ast$ is a chain complex. That is, the boundary maps of~$(C_n)_\ast$ satisfy $\partial^{p-1}\circ \partial^p=0$.
\end{lem}

\begin{proof}
    We show that iterated differentials vanish, by observing that if~$p\geqslant 1$ and $0\leqslant j<k\leqslant p$, then the composite maps~$d^{p-1}_jd^p_k,d^{p-1}_{k-1}d^{p}_j\colon (C_n)_p\to (C_n)_{p-2}$ coincide.  
    It follows that~$\partial^{p-1}\circ 
    \partial^p=0$.
    We have
    \[
        d^{p-1}_jd^p_k(x\otimes r) = [x\cdot(S_{n-i+k-1}\cdots S_{n-i})\cdot (S_{n-i+j}\cdots S_{n-i+1})]\otimes r
    \]
    and
    \[
        d^{p-1}_{k-1}d^{p}_j(x\otimes r) = [x\cdot(S_{n-i+j-1}\cdots S_{n-i})\cdot (S_{n-i+k-1}\cdots S_{n-i+1})]\otimes r.
    \]
    By repeated use of the braiding relations on the~$S_m$, these maps coincide. (For a more detailed proof compare with~\cite[Lemma 4.8]{HepworthBoydStability}.)
\end{proof}

The rest of this section is devoted to proving that  this chain complex is highly acyclic, which is made precise in the following theorem.

\begin{thm}\label{thm:Brauer complex}
$H_i(C_n) = 0$ for $i \le \frac{n-3}2$.
\end{thm}

To prove this theorem, we introduce a number of decompositions, filtrations, and chain isomorphisms, which eventually reduce the theorem to high acyclicity of the complex of injective words.

\subsection{A decomposition and a filtration}

We first observe that for a single diagram, corresponding to a monomial in~$\Br_n$, the number of left-to-left connections (which is equal to the number of right-to-right connections) is invariant under the differential $\partial$. This is because the number of left-to-left connections remains invariant under multiplication by any~$S_i$, and under the operation of extending the box over a new node. 

\begin{Def}\label{def-decomposition}
Fix~$0\leq k \leq \lfloor n/2\rfloor$. Let~$C_n^{(k)}=(C_n^{(k)})_\ast$ be the subcomplex of the chain complex~$C_n$ with basis at degree~$p$ consisting of those diagrams in~$(C_n)_p$ which have~$k$ left-to-left connections.
\end{Def}

Thus we can decompose~$C_n$ as
\[ C_n = C_n^{(0)} \oplus C_n^{(1)} \oplus \dots \oplus C_n^{(\lfloor n/2\rfloor)}.\]

\begin{prop}\label{prop:Brauer complex level (k)}
$H_i(C_n^{(k)}) = 0$ for all $i \le n-k-2$.
\end{prop}

Assuming the above proposition we prove \autoref{thm:Brauer complex} as a consequence.

\begin{proof}[Proof of \autoref{thm:Brauer complex}]
    If $i\le\frac{n-3}{2}$ as in the statement of \autoref{thm:Brauer complex}, then
    \[
        i \le \left\lfloor \frac{n-3}2 \right\rfloor = n - \left\lfloor \frac n2 \right\rfloor -2 \le n - k -2
    \]
    for all $0\le k\le \lfloor\frac{n}{2}\rfloor$. 
    Then, by the decomposition of~$C_n$ into the~$C_n^{(k)}$,
    \[ H_i(C_n) = \bigoplus_{k= 0}^{\lfloor n/2\rfloor} H_i(C_n^{(k)})\]
    vanishes by \autoref{prop:Brauer complex level (k)}.
\end{proof}

We are left to prove \autoref{prop:Brauer complex level (k)}. 
Observe that if there are~$k$ left-to-left connections in a diagram, it follows that there are~$k$ right-to-right connections. However, on the right hand side of a diagram in~$(C_n)_p$ there is a $(n-(p+1))$-box, and so the right-to-right connections are split into two sets: singular nodes connected to the box and pairs of nodes connected to each other and not connected to the box. We exploit this in a filtration of~$(C_n^{(k)})_\ast$.

\begin{Def}\label{def-filtration}
For each $0\leq k \leq \lfloor n/2\rfloor$, we define a filtration 
\[
    F_0C_n^{(k)}
    \subseteq
    F_1C_n^{(k)}
    \subseteq
    \cdots
    \subseteq
    F_kC_n^{(k)}
\]
of $C_n^{(k)}$ as follows. 
The~$j$th level $F_jC_n^{(k)}$ is generated by diagrams with at most $j$ right-to-right connections that are \emph{not} connected to the box. Note that this is indeed a filtration, since the defining criterion is invariant under $\partial$. This is due to the observation that the boundary map can only decrease the number of right-to-right connections {not} connected to the box.
\end{Def}

\begin{prop}\label{prop:filtration quotient acyclicity}
The filtration quotient $F_jC_n^{(k)}/F_{j-1}C_n^{(k)}$ is highly acyclic: its homology vanishes in degrees $i\le n-k-2+j$.
\end{prop}

The proof of this result will be given later.
Assuming it for the time being, we may prove \autoref{prop:Brauer complex level (k)}:

\begin{proof}[Proof of \autoref{prop:Brauer complex level (k)}]
    By \autoref{prop:filtration quotient acyclicity}, the homology of the filtration quotient $F_jC_n^{(k)}/F_{j-1}C_n^{(k)}$ vanishes in degrees $\ast\leq n-k-2$ for all~$j$.
    The same then holds for $C_n^{(k)}$ itself, either by considering 
    the spectral sequence associated to the filtration, or by considering the long exact sequences associated to the short exact sequences $0\to F_{j-1}C_n^{(k)}\to F_jC_n^{(k)}\to F_jC_n^{(k)}/F_{j-1}C_n^{(k)}\to 0$.
\end{proof}

\subsection{Injective words with separators}

Our goal is now to prove~\autoref{prop:filtration quotient acyclicity}.
We will do this at the end of this section by identifying the filtration quotients in terms of yet another family of complexes, which we introduce now.

\begin{Def}[Injective words with separators]
    Let $X$ be a finite set and let $k\geqslant 0$.
    An \emph{injective word on $X$ with $k$ separators} is a word with letters taken from the set $X\sqcup\{|\}$ consisting of $X$ and the \emph{separator} $|$, where each letter from $X$ appears at most once, and where the separator appears exactly $k$ times.
    When $k=0$, then these are simply the injective words on $X$.
\end{Def}

\begin{ex}
    For example, if $X=\{a\}$ and $k=2$, then the possible words are:
    \[
        ||
        \qquad\qquad
        ||a
        \qquad\qquad
        |a|
        \qquad\qquad
        a||
    \]
    And if $X=\{a,b\}$ and $k=1$, then the possible words are:
    \[
        |
        \qquad
        |a
        \qquad
        a|
        \qquad
        |b
        \qquad
        b|
        \qquad
        |ab
        \qquad
        a|b
        \qquad
        ab|
        \qquad
        |ba
        \qquad
        b|a
        \qquad
        ba|
    \]
\end{ex}

\begin{Def}[The complex of injective words with separators]
    Let $X$ be a finite set, let $s\geqslant 0$, and let $R$ be a commutative ring.
    The \emph{complex of injective words with $s$ separators} is the $R$-chain complex $W_X^{(s)}$ concentrated in degrees $-1\leq p \leq |X|-1$, and defined as follows.
    In degree $p$, $(W^{(s)}_X)_p$ has basis given by the injective words on $X$ with $s$ separators with $(p+1)$ letters from $X$. 
    Thus such a word ${\bf a} \in (W^{(s)}_X)_p$ has length~$s+p+1$. Let~$r=s+p$ and~${\bf a}=a_0a_1\cdots a_r$.
    The boundary operator $\partial^p\colon (W^{(s)}_X)_p\to (W^{(s)}_X)_{p-1}$ is defined by the rule
    \[
        \partial^p(a_0a_1\cdots a_r)=\sum_{i=0}^r(-1)^ia_0\cdots\widehat{a_i}\cdots a_r
    \]
    subject to the condition that if the omitted letter is a separator, then the corresponding term is omitted (or identified with $0$).
    In other words, the boundary is the signed sum of the words obtained by deleting the letters that come from $X$ and \emph{not deleting any separators}, but with signs determined by the position of the deleted letter among all letters \emph{including the separator}:
    \[
        \partial^p(a_0a_1\cdots a_r)=\sum_{a_i\in X}(-1)^ia_0\cdots\widehat{a_i}\cdots a_r
    \]
\end{Def}

\begin{ex}
    If we take $X=\{a,b\}$ and $s=1$, then the elements $a|b$ and $|ba$ both live in degree $1$, and their boundaries are
    \[
        \partial^1(a|b)
        = |b + a|
        \qquad\text{and}\qquad
        \partial^1(|ba) = - |a + |b.
    \]
\end{ex}

\begin{lem}
    $W^{(s)}_X$ is a chain complex.
    That is, $\partial^{p-1}\circ\partial^p=0$.
\end{lem}

\begin{proof}
    If $a_i,a_j\in X$ are distinct letters of a word $a_0\cdots a_r\in (W^{(s)}_X)_i$, with $i<j$, then $a_0\cdots \widehat a_i\cdots\widehat a_j\cdots a_r$ appears in $\partial^{p-1}\partial^p(a_0\cdots a_r)$ with coefficient \[(-1)^{i+(j-1)} + (-1)^{j+i}=0. \qedhere\]
\end{proof}

\begin{prop}\label{prop:Wnk highly acyclic}
$W_X^{(s)}$ is highly acyclic: $H_i(W_X^{(s)})=0$ for~$i\leq |X|-2$.
\end{prop}

\begin{proof}
We prove the statement by induction on $s$. 
The base case $s=0$ is the usual acyclicity statement for the complex of injective words, first proved by Farmer~\cite{Farmer}.

Let $s\ge 1$ and assume the result holds true for~$(s-1)$. 
We introduce a filtration $F_jW_X^{(s)}$ which at level~$j$ is generated by words $a_0\cdots a_r$ in which a separator appears among the letters $a_0,\ldots,a_j$, or equivalently, in at least one of the positions $0,\ldots,j$.
Then the filtration quotient $F_jW_X^{(s)}/F_{j-1}W_X^{(s)}$ is the chain complex defined like $W_X^{(s)}$, but with basis consisting of words in which the first separator appears in position $j$, and whose differential omits letters taken from $X$ that appear \emph{after} position $j$.  

Next we introduce a degree $-j$ chain map
\[
    \Psi_\ast\colon F_j W_X^{(s)}/F_{j-1} W_X^{(s)}
    \longrightarrow
    \bigoplus_{\mathbf{x}} W^{(s-1)}_{X-\mathbf{x}}
\]
where $\mathbf{x}$ ranges over all injective words on $X$ with exactly $j$ letters (and no separators), and $X-\mathbf{x}$ denotes the set obtained by deleting the letters of $\mathbf{x}$ from $X$.
Observe that if  $a_0\cdots a_r$ is a basis element of  $(F_j W_X^{(s)})_p / (F_{j-1} W_X^{(s)})_p$ then $a_0\cdots a_r = (a_0\cdots a_{j-1}) | (a_{j+1}\cdots a_r)$ where $a_0\cdots a_{j-1}$ is an injective word on $X$ with no separators, and $a_{j+1}\cdots a_r$ is an injective word of degree $p-j$ on $X-\mathbf{x}$ with $s-1$ separators.
Then $\Psi_\ast$ is defined by 
\[
    \Psi_p(a_0\cdots a_r) = (-1)^{(j+1)p}a_{j+1}\cdots a_r \in W^{(s-1)}_{X-(a_0\cdots a_{j-1})}.
\]
It is straightforward to see that $\Psi_\ast$ is both a chain map and an isomorphism.

Since the homology of each $W^{(s-1)}_{X-\mathbf{x}}$ vanishes in degrees $\ast\leq |X-\mathbf{x}|-2=|X|-j-2$, it follows that the homology of $F_j W_X^{(s)}/F_{j-1} W_X^{(s)}$ vanishes in the same range.
The same therefore holds for $W_X^{(s)}$ itself, either by considering the spectral sequence associated to the filtration, or by considering the long exact sequences associated to the short exact sequences $0\to F_{j-1}W_X^{(s)}\to F_jW_X^{(s)}\to F_jW_X^{(s)}/F_{j-1}W_X^{(s)}\to 0$.
\end{proof}

\subsection{Identifying the filtration quotients}

Recall from \autoref{def-decomposition} that $C_n^{(k)}\subseteq C_n$ is the subcomplex spanned by diagrams that have precisely $k$ left-to-left connections,
and from \autoref{def-filtration} that $F_jC_n^{(k)}\subseteq C_n^{(k)}$ is the subcomplex spanned by diagrams that have at most $j$ right-to-right connections that are not connected to the box.
Our aim now is to prove \autoref{prop:filtration quotient acyclicity}, which states that the filtration quotients  $F_jC_n^{(k)}/F_{j-1}C_n^{(k)}$ are highly acyclic.
We will do this by identifying the filtration quotients in terms of complexes of injective words with separators.

To begin, observe that the filtration quotient $F_jC_n^{(k)}/F_{j-1}C_n^{(k)}$ has basis in degree~$p$ consisting of diagrams which have an~$(n-(p+1))$-box on the right,~$k$ left-to-left connections and~$k$ right-to-right connections, of which exactly $j$ are not connected to the box. 
In more detail, such diagrams have the following properties:
\begin{itemize}
    \item
    On the left, the diagram has $n$ nodes.
    \item
    There are precisely $k$ left-to-left connections.
    \item
    On the right, the diagram has an $(n-(p+1))$-box, together with $(p+1)$ nodes.
    \item
    There are exactly $j$ right-to-right connections that do not connect to the box, and a further $k-j$ right-to-right connections that connect a node to the box.
    (There are $k$ left-to-left connections, so there are also $k$ right-to-right connections.)
    \item
    Thus $(k-j)+2j = k+j$ of the right-hand nodes are part of right-to-right connections.
    \item
    The remaining $(p+1)-(k+j)$ right-hand nodes are therefore connected to nodes on the left.
\end{itemize}

\begin{ex}\label{ex:quotient diagram}
Here is an example of a diagram in $F_jC_n^{(k)}/F_{j-1}C_n^{(k)}$ for $n=8$, $k=2$ and $j=1$, lying in degree $p=5$.
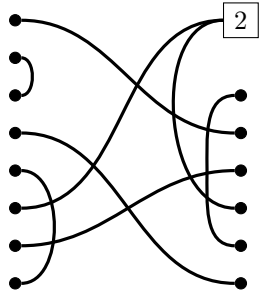
\[
    \begin{tikzpicture}[x=1.5cm,y=-.5cm,baseline=-2.05cm]
    
        \node[v] (a1) at (0,0) {};
        \node[v] (a2) at (0,1) {};
        \node[v] (a3) at (0,2) {};
        \node[v] (a4) at (0,3) {};
        \node[v] (a5) at (0,4) {};
        \node[v] (a6) at (0,5) {};
        \node[v] (a7) at (0,6) {};
        \node[v] (a8) at (0,7) {};
        
        \node[b] (b1) at (2,0) {$2$};
        \node[v] (b3) at (2,2) {};
        \node[v] (b4) at (2,3) {};
        \node[v] (b5) at (2,4) {};
        \node[v] (b6) at (2,5) {};
        \node[v] (b7) at (2,6) {};
        \node[v] (b8) at (2,7) {};
        
        \draw[e] (a1) to[out=0, in=180] (b4);
        \draw[e] (a4) to[out=0, in=180] (b8);
        \draw[e] (a6) to[out=0, in=180] (b1);
        \draw[e] (a7) to[out=0, in=180] (b5);
        \draw[e] (a2) to[out=0, in=0]   (a3);
        \draw[e] (b1) to[out=180, in=180] (b6);
        \draw[e] (a5) to[out=0, in=0] (a8);
        %\draw[e] (b3) to[out=180,in=180] (b7);
        \draw[e] (b3) to[out=180,in=90](1.7,4) to[out=270,in=180] (b7);
    
    \end{tikzpicture}
\]
Observe the $k=2$ left-to-left connections, the $k=2$ right-to-right connections, the $j=1$ right-to-right connection that does not involve the box, the $k+j=3$ nodes involved in right-to-right connections, and the $(p+1)-(k+j)=3$ nodes on the right that are connected to nodes on the left.
\end{ex}

\begin{Def}\label{def:4-tuple}
A diagram in the basis of $(F_jC_n^{(k)}/F_{j-1}C_n^{(k)})_p$ determines a tuple
\[
    (X,P,Y, \mathbf{a})
\]
consisting of the following data:
\begin{itemize}
    \item
    A subset $X\subseteq[n]$ of cardinality $n-2k$.
    \item
    A partition $P$ of $[n]-X$ into $k$ disjoint pairs.
    \item
    An ordered set $Y$ of cardinality $k+j$, with $2j$ of its elements divided into disjoint pairs. This is regarded as an isomorphism class of such ordered sets with pairing.
    \item
    An injective word with $k+j$ separators $\mathbf{a}$, of degree $p-(k+j)$, on the set $X$.
\end{itemize}
The diagram determines the tuple as follows:
\begin{itemize}
    \item
    $X\subseteq[n]$ is the set of nodes on the left that are part of left-to-right connections.
    \item
    $P$ is the partition of $[n]-X$ determined by the left-to-left connections.
    \item
    $Y$ is the set of nodes on the right that are part of right-to-right connections, either to other nodes or to the box, with the pairing obtained from the right-to-right connections that do not involve the box.
    \item
    $\mathbf{a}$ is the injective word of length $p+1$ obtained as follows.
    If the $i$-th node on the right is part of a right-to-right connection, then the $i$-th letter of $\mathbf{a}$ is a separator.  
    If the $i$-th node on the right is part of a left-to-right connection, then the $i$-th letter of $\mathbf{a}$ is the element of $X$ at the left-hand end of that connection.
\end{itemize}    
\end{Def}

\begin{ex}\label{ex:diagram tuple}
Continuing from \autoref{ex:quotient diagram}, the tuple $(X,P,Y,\mathbf{a})$ associated to our previous diagram is depicted below: 
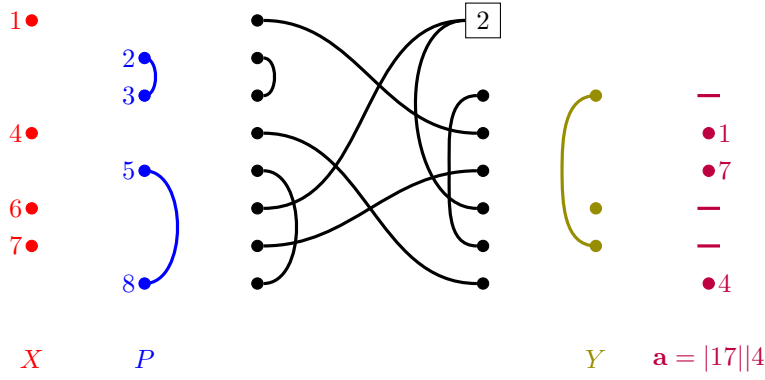
\[
    \begin{tikzpicture}[x=1.5cm,y=-.5cm,baseline=-2.05cm]
    
        \node[v] (a1) at (0,0) {};
        \node[v] (a2) at (0,1) {};
        \node[v] (a3) at (0,2) {};
        \node[v] (a4) at (0,3) {};
        \node[v] (a5) at (0,4) {};
        \node[v] (a6) at (0,5) {};
        \node[v] (a7) at (0,6) {};
        \node[v] (a8) at (0,7) {};
        
        \node[b] (b1) at (2,0) {$2$};
        \node[v] (b3) at (2,2) {};
        \node[v] (b4) at (2,3) {};
        \node[v] (b5) at (2,4) {};
        \node[v] (b6) at (2,5) {};
        \node[v] (b7) at (2,6) {};
        \node[v] (b8) at (2,7) {};
        
        \draw[e] (a1) to[out=0, in=180] (b4);
        \draw[e] (a4) to[out=0, in=180] (b8);
        \draw[e] (a6) to[out=0, in=180] (b1);
        \draw[e] (a7) to[out=0, in=180] (b5);
        \draw[e] (a2) to[out=0, in=0]   (a3);
        \draw[e] (b1) to[out=180, in=180] (b6);
        \draw[e] (a5) to[out=0, in=0] (a8);
        %\draw[e] (b3) to[out=180,in=180] (b7);
        \draw[e] (b3) to[out=180,in=90](1.7,4) to[out=270,in=180] (b7);
        
        \node[v, red] at (-2,0) {};
        \node[left,red] at (-2,0){$1$};
        \node[v,red] at (-2,3) {};
        \node[left,red] at (-2,3){$4$};
        \node[v,red] at (-2,5) {};
        \node[left,red] at (-2,5){$6$};
        \node[v,red] at (-2,6) {};
        \node[left,red] at (-2,6){$7$};
        
        \node[v,blue] at (-1,1) {};
        \node[left,blue] at (-1,1){$2$};
        \node[v,blue] at (-1,2){};
        \node[left,blue] at (-1,2){$3$};
        \node[v,blue] at (-1,4){};
        \node[left,blue] at (-1,4){$5$};
        \node[v,blue] at (-1,7){};
        \node[left,blue] at (-1,7){$8$};
        \draw[e,blue] (-1,1) to[out=0, in=0] (-1,2);
        \draw[e,blue] (-1,4) to[out=0, in=0] (-1,7);
        
        \node[v,olive] at (3,2) {};
        \node[v,olive] at (3,5) {};
        \node[v,olive] at (3,6) {};
        \draw[e,olive] (3,2) to[out=180,in=90] (2.7,4) to[out=270,in=180] (3,6);
    
        \node[v,purple] at (4,3) {};
        \node[right,purple] at (4,3) {$1$};
        \node[v,purple] at (4,4) {};
        \node[right,purple] at (4,4) {$7$};
        \node[v,purple] at (4,7) {};
        \node[right,purple] at (4,7) {$4$};
        \draw[e,purple] (3.9,2) to (4.1,2);
        \draw[e,purple] (3.9,5) to (4.1,5);
        \draw[e,purple] (3.9,6) to (4.1,6);
        
        \node[red] at (-2,9) {$X$};
        \node[blue] at (-1,9) {$P$};
        \node[olive] at (3,9) {$Y$};
        \node[purple] at (4,9) {$\mathbf{a}=|17||4$};
    \end{tikzpicture}
\]
Thus $X$ consists of all left-hand nodes that are part of left-to-right connections, and $P$ is the pairing obtained from the left-to-left connections on the remaining left-hand nodes. 
We have equipped $X$ and $P$ with the labelling obtained by labelling the nodes from top to bottom, in order to indicate that $X$ is a subset of $[n]=[8]$ and that $P$ is a pairing on the complement of $X$.
Next, $Y$ consists of the nodes that are part of right-to-right connections, equipped with the pairing obtained from the right-to-right connections that do not involve the box.
We have not labelled $Y$, in order to indicate that it is only the isomorphism type of $Y$, as an ordered set with pairing, which is recorded.
Finally, $\mathbf{a}$ is drawn by writing the ends of right-to-right connections as a horizontal bar (separator), together with the nodes that are the right-hand ends of left-to-right connections, labelled by the element at the left-hand end of the relevant connection.
Rotating this by $90$ degrees gives us the word $\mathbf{a}$, in this case $\mathbf{a}=|17||4$.
Note that the letter $6\in X$ does not appear in $\mathbf{a}$.  This is because in the original diagram, node $6$ on the left is connected to the box.

Finally, observe that we can completely rebuild the original diagram from the tuple $(X,P,Y,\mathbf{a})$.
First, $P$ allows us to recover all left-to-left connections.
Next, matching $Y$ to the separators of $\mathbf{a}$ allows us to recover the nodes on the right, together with the right-to-right connections between the nodes.
Finally, right-hand nodes that are labelled by an element of $X$ are then equipped with a left-to-right connection to that element,
and any remaining right-hand nodes are then connected to the box.
\end{ex}

\begin{Def}
    The discussion above allows us to define a map
    \[
        \Phi_\ast
        \colon
        F_jC_n^{(k)}/F_{j-1}C_n^{(k)}
        \longrightarrow
        \bigoplus_{(X,P,Y)} W^{(k+j)}_X
    \] 
    of degree $-(k+j)$, where the direct sum is indexed by all tuples $(X,P,Y)$ of the form considered in \autoref{def:4-tuple}.
    The map $\Phi_p$ is defined as follows:
    Take a diagram $D$ in the basis of   $(F_jC_n^{(k)}/F_{j-1}C_n^{(k)})_p$, determine the associated tuple $(X,P,Y,\mathbf{a})$ as in \autoref{def:4-tuple}, and then define $\Phi_\ast(D)$ to be $\mathbf{a}$ in the summand $W^{(k+j)}_X$ corresponding to $(X,P,Y)$.
\end{Def}

\begin{lem}
    $\Phi_\ast$ is a chain map.
\end{lem}

\begin{proof}
    We suggest that the reader keep in mind the diagram from \autoref{ex:diagram tuple} throughout this proof.
    
    First, we claim that the triple $(X,P,Y)$ associated, via~$\Phi_\ast$,  to a basis diagram $D$ in $(F_jC_n^{(k)}/F_{j-1}C_n^{(k)})_p$ is preserved in all diagrams appearing in the boundary of $D$.
    Recall from \autoref{def:chain complex} that the boundary map $\partial^p$ sends a diagram to the alternating sum of the diagrams obtained as follows: work through the nodes on the right of the diagram, and in each case move the node into the box.
    This clearly does not change the left-hand end of the diagram, and therefore all of the diagrams in the boundary have the same $X$ and $P$ associated to them.
    If the node that is moved into the box was part of a right-to-right connection with another node, then the resulting diagram has fewer right-to-right connections between nodes, and therefore vanishes in the quotient 
    $(F_jC_n^{(k)}/F_{j-1}C_n^{(k)})_p$.
    If the node that is moved into the box was part of a right-to-right connection with the box, then after moving it into the box, the resulting diagram has a loop at the box, and therefore again vanishes.
    The last two sentences show that, under the boundary map, the nodes that are part of the right-to-right connections do not change, and neither does the data of which are connected to the box and which are not.
    In other words, the diagrams in the boundary have the same $Y$ associated to them.
    
    The above paragraph demonstrates that $F_jC_n^{(k)}/F_{j-1}C_n^{(k)}$ splits as a direct sum indexed by the triples $(X,P,Y)$.
    It now suffices to show that the assignment that sends a diagram with fixed $(X,P,Y)$ to the corresponding injective word with separators $\mathbf{a}$ respects the boundary map.
    But this is clear: moving the end of a left-to-right connection into the box corresponds exactly to deleting one of the non-separator letters from $\mathbf{a}$.
\end{proof}

\begin{lem}\label{lem:Phi iso}
    $\Phi_\ast$ is an isomorphism.
\end{lem}

\begin{proof}
    We work in degree $p$.  
    The domain of $\Phi_\ast$ has a basis consisting of the diagrams in degree $p$, while the range has basis consisting of all $4$-tuples $(X,P,Y,\mathbf{a})$ of the kind appearing in \autoref{def:4-tuple}, and the effect of $\Phi_\ast$ is to send a diagram to the associated tuple.
    It therefore suffices to show that the assignment that sends a basis diagram to a $4$-tuple is a bijection.
    
    To see this, we construct an inverse assignment as follows.
    We again recommend that the reader keeps the diagram of \autoref{ex:diagram tuple} in mind at this point.
    Given a tuple $(X,P,Y,\mathbf{a})$, we build a basis diagram, as follows.
    First, $P$ allows us to recover all left-to-left connections on $[n]$.
    Next, we use the letters of $\mathbf{a}$ --- including the separators --- as the right-hand nodes of our diagram.
    We then replace the separators, in order, with the elements of $Y$ (which is linearly ordered), adding connections according to the pairing on $Y$, and connecting any unpaired elements of $Y$ to the box.
    The unconnected nodes on the right are then the non-separators of $\mathbf{a}$, which are in fact distinct elements of $X$; adding the corresponding left-to-right connections then completes the diagram.
    
    The assignment described in the last paragraph is indeed an inverse to the one described in the first paragraph, and this completes the proof.
\end{proof}

Using the isomorphism~$\Phi_\ast$, we can subsequently prove~\autoref{prop:filtration quotient acyclicity}.

\begin{proof}[Proof of \autoref{prop:filtration quotient acyclicity}]

We wish to show that the homology of $F_jC_n^{(k)}/F_{j-1}C^{(k)}_n$ vanishes in degrees $i\leq n-k-2+j$.
By \autoref{lem:Phi iso}, there is a chain isomporphism of degree $-(k+j)$ between this chain complex and $\bigoplus_{(X,P,Y)}W_X^{(k+j)}$, where in this direct sum the sets $X$ all have cardinality $n-2k$.
By \autoref{prop:Wnk highly acyclic}, the homology of the $W_X^{(k+j)}$ all vanish in degrees $i\leq |X|-2 = n-2k-2$, so that the homology of $F_jC_n^{(k)}/F_{j-1}C^{(k)}_n$ vanishes in degrees $i\leq n-2k-2 + (k+j) = n-k-2+j$. 
\end{proof}

\section{Proof of \autoref{thma}}\label{Section: MainTheorem}

In this section, we prove \autoref{thma}. Recall that this states that the inclusion $\iota: R\mathfrak S_n \hookrightarrow \Br_n$ induces maps
\[ H_i(\mathfrak S_n;\t) \longrightarrow \Tor^{\Br_n}_i(\t,\t) \]
that are isomorphisms for $n\ge 2i+1$.

Recall the inclusion and projection maps
\[
    R\fS_n\xrightarrow{\iota}\Br_n\xrightarrow{\pi} R\fS_n,
\]
which satisfy $\pi\circ\iota=\mathrm{id}$.
These induce maps of $\Tor$-groups
\[
    \Tor^{R\fS_n}(\t,\t)\xrightarrow{\iota_\ast}\Tor^{\Br_n}(\t,\t)\xrightarrow{\pi_\ast} \Tor^{R\fS_n}(\t,\t)
\]
that again satisfy $\pi_\ast\circ\iota_\ast=\mathrm{id}$, so that 
$\iota_\ast$ is injective.
It is therefore enough to show that $\iota_\ast$ is surjective.

Recall the chain complex~$C_n=(C_n)_\ast$ from \autoref{def:chain complex}, and recall this complex is non-zero only when~$-1\leq \ast \leq n-1$. Recall from \autoref{thm:Brauer complex} that $H_i(C_n)=0$ for~$i\leq \frac{n-3}{2}$.

Consider the two spectral sequences associated to the double complex~$P_\ast \otimes_{\Br_n} (C_n)_\ast$, for~$P_\ast$ a projective resolution of~$\Br_n$ over~$\t$. The first spectral sequence converges to zero in a range, since each~$P_p$ is projective, and therefore flat, so
\[
E^1_{p,q}=H_q(P_p\otimes_{\Br_n} (C_n)_\ast)\cong P_p\otimes_{\Br_n} H_q((C_n)_\ast)=0 \mbox{ if } q\leq \frac{n-3}{2}.
\]
Therefore the second spectral sequence will also converge to zero in the range $p+q \le \frac{n-3}2$. The~$E^1$-page of this spectral sequence can be identified as follows
\begin{eqnarray*}
E^1_{p,q}&=&\Tor_q^{\Br_n}(\t, (C_n)_p)\\
&=&\Tor_q^{\Br_n}(\t, \Br_n\otimes_{\Br_{n-p-1}}\t)\\&\cong&
\begin{cases}
\Tor_q^{\Br_n}(\t,\t)& p=-1\\ H_q(\fS_{n-p-1};\t) & p\geq 0. \end{cases}
\end{eqnarray*}
where the final line uses \autoref{prop-tor-quotient-Homology-Sm}.
Our aim now is to identify the differentials, but we include a theoretical Lemma first, to help us do this.

\begin{lem}\label{lem-induced-identity}
    Let $G$ be a group, $H$ a subgroup of $G$, and $g\in G$ an element that commutes with $H$.
    Right-multiplication by $g$ induces a map $RG\otimes_{RH}\t\to RG\otimes_{RH}\t$.
    The induced map on $\Tor$ groups
    \[
        \Tor^{RG}_\ast(\t,RG\otimes_{RH}\t)
        \longrightarrow
        \Tor^{RG}_\ast(\t,RG\otimes_{RH}\t).
    \]
    is the identity map.
\end{lem}

\begin{proof}
    Let $P_\ast$ be a projective $RG$-resolution of $\t$.
    Then it is also a projective $RH$-resolution of $\t$.
    Since $g$ commutes with $RH$, the map $P_\ast\to P_\ast$, $p\mapsto pg$ is a map of projective $RH$-resolutions, over the identity map on $\t$.
    It is therefore chain-homotopic to the identity map of $P_\ast$.
    
    $\Tor^{RG}_\ast(\t,RG\otimes_{RH}\t)$ is the homology of $P_\ast\otimes_{RG}(RG\otimes_{RH}\t)$, which we identify with $P_\ast\otimes_{RH}\t$, and then the map in question is represented by the chain map $P_\ast\otimes_{RH}\t\to P_\ast\otimes_{RH}\t$, $p\otimes r\mapsto pg\otimes r$.
    This is chain-homotopic to the identity, by the first paragraph.
\end{proof}

We now return to considering the differentials. Recall the map~$\sigma$ from the beginning of \autoref{Section:Shapiro}, which is an isomorphism by Shapiro's Lemma. Recall also that the map~$\iota_
\ast\colon\Tor^{R\fS_n}_\ast(\t,R\fS_n\otimes_{R\fS_m}\t)\to \Tor^{\Br_n}_\ast(\t,\Br_n\otimes_{\Br_m}\t)$ induced by the inclusion~$\iota$ is an isomorphism by \autoref{prop-tor-quotient-Homology-Sm}. 

\begin{lem}\label{lemma-d-one}
    Under the isomorphisms  
    \[ 
        \iota_\ast\circ\sigma\colon  
        H_\ast(\fS_{n-p-1})
        \xrightarrow{\ \cong\ } 
        \Tor_q^{\Br_n}(\t,\Br_n\otimes_{\Br_{n-p-1}}\t)
    \]
    for $p\geq 0$, the $d^1$-differentials in the above spectral sequence are given as follows:
    \[
        d^1\colon E^1_{p+1,q}\to E^1_{p,q}
        =\begin{cases} 
            0 
            & 
            p\geq 0 \text{ even}
            \\
            s_\ast\colon H_q(\fS_{n-p-2})\to H_q(\fS_{n-p-1})
            & 
            p\geq 0 \text{ odd}
            \\
            s_\ast\colon H_q(\fS_{n-1})\to \Tor_q^{\Br_n}(\t,\t) & p=-1
        \end{cases}
    \]
    where $s\colon \fS_{n-p-2}\to\fS_{n-p-1}$ and $s\colon R\fS_{n-1}\to\Br_n$ are the inclusion maps.
\end{lem}

\begin{proof}
Recall that the differential of~$(C_n)_\ast$ originating at~$(C_n)_{p+1}$ is a composite
\[
    \Br_n \tens[\Br_{n-(p+2)}] \t 
    \longrightarrow 
    \Br_n \tens[\Br_{n-(p+2)}] \t 
    \longrightarrow
    \Br_n \tens[\Br_{n-(p+1)}] \t.
\]
The first map is the alternating sum of the maps induced by right multiplication by the symmetric group elements~$\sigma^{p+1}_i:=(S_{n-(p+1)+i-1}\cdots S_{n-(p+1)})$ for $i=0,\ldots,p+1$. 
The second map extends the tensor from~$\Br_{n-(p+2)}$ to~$\Br_{n-(p+1)}$.
The~$d^1$-differential originating at~$E^1_{p+1,q}$ is the induced composite, obtained by applying $\Tor^{\Br_n}_q(\t,-)$ to the differential of $(C_n)_\ast$.

Let $p\geq 0$.
Under the isomorphism $\iota_\ast$, $d^1$ becomes the entirely analogous composite with each $\Br_\ell$ replaced by the corresponding $R\fS_\ell$.
The first factor in the resulting composite is, by \autoref{lem-induced-identity}, the alternating sum of $(p+2)$ identity maps, and is therefore the identity for $p$ odd and zero for $p$ even.
Applying the isomorphism $\sigma$ from Shapiro's lemma now gives us the required description, where~$s$ stands for the \emph{stabilisation map} i.e.~$s_
\ast$ is induced by the inclusion $\fS_{n-p-2}\hookrightarrow \fS_{n-p-1}$.

Let $p=-1$.
Then the differential of $(C_n)_\ast$ originating at $(C_n)_0$ is simply the map
\[
    \varepsilon\colon \Br_n\otimes_{\Br_{n-1}}\t
    \to
    \t,
    \qquad
    x\otimes r\mapsto x\cdot r.
\]
The $d^1$ differential originating at $E_{0,q}$ is the induced map
\[
    \varepsilon_\ast\colon
    \Tor^{\Br_n}_q(\t,\Br_n\otimes_{\Br_{n-1}}\t)
    \longrightarrow
    \Tor^{\Br_n}_q(\t,\t).
\]
Under the isomorphism $\iota_\ast\circ\sigma$, this map is replaced with the composite $\varepsilon_\ast\circ\iota_\ast\circ \sigma$.
This can be computed explicitly, and gives the result.
\end{proof}

\begin{thm}\label{thm:BncongSn-1}
    In the range~$i\leq \frac{n-1}{2}$, the inclusion map $R\fS_{n-1}\hookrightarrow \Br_n$
    induces a surjection
    \[
     H_i(\fS_{n-1};\t)\longrightarrow \Tor_i^{\Br_n}(\t,\t).
    \]
\end{thm}
    
\begin{proof}
    Consider the spectral sequence associated to the double complex $P_\ast \otimes_{\Br_n} (C_n)_\ast$ as discussed above, for~$P_\ast$ a projective resolution of~$\Br_n$ over~$\t$. Then the~$E^1$-page of the spectral sequence is as shown in \autoref{fig:spectralsequence}.
    By \autoref{lemma-d-one}, it will suffice to prove that the map $d^1\colon E^1_{0,q}\to E^1_{-1,q}$ is a surjection in degrees $q\leq \frac{n-1}{2}$.
        
 \begin{figure}
		\begin{tikzpicture}[scale=0.75]
			\draw[line width=0.2cm, gray!20, <->] (-4,7)--node[black, above, pos=0] {$q$}(-4,-1)--(12,-1) node[black, right] {$p$};
			\draw (-4,0) node {$0$};
			\draw (-2,-1) node {$-1$};
			\draw (1,-1) node {$0$};		
			\draw (-4,1) node {$1$};
			\draw (-4,2) node {$2$};
			\foreach \x in {-4,-2, 1, 4.5, 6, 8.5,11}	
			\draw (\x,3) node {$\vdots$};
			\foreach \x in {-1,0,1,2,4,5}
			\draw (7,\x) node {$\cdots$};	
			\draw (4.5,-1) node {$1$};	
			\draw (11,-1) node[rotate=90]  {$\scriptstyle{n}$};				
			\draw(8.5,-1) node[rotate=90]  {$\scriptstyle{n-1}$};
			\draw (-4.1,4) node {$\scriptstyle{\lfloor\frac{n-3}{2}\rfloor}$};
			\draw (-4.1,5) node {$\scriptstyle{\lfloor\frac{n-1}{2}\rfloor}$};
			\draw (-4,6) node {$\scriptstyle{\lfloor\frac{n+1}{2}\rfloor}$};
			\foreach \x in {1,2}
			\draw (-2,\x) node {$\scriptstyle{\Tor^{\Br_n}_{\x}(\t,\t)}$};
			\draw (-2,4) node {$\scriptstyle{\Tor^{\Br_n}_{\lfloor\frac{n-3}{2}\rfloor}(\t,\t)}$};
			\draw (-2,5) node {$\scriptstyle{\Tor^{\Br_n}_{\lfloor\frac{n-1}{2}\rfloor}(\t,\t)}$};
			\draw (-2,6) node {$\scriptstyle{\Tor^{\Br_n}_{\lfloor\frac{n+1}{2}\rfloor}(\t,\t)}$};
			\draw (-2,0) node {$\scriptstyle{\Tor_0^{\Br_n}(\t, \t)}$};
			\draw (1.25, 4) node {$\scriptstyle{H_{\lfloor\frac{n-3}{2}\rfloor}(\fS_{n-1})}$} (1.25, 5) node {$\scriptstyle{H_{\lfloor\frac{n-1}{2}\rfloor}(\fS_{n-1})}$};
			\draw (4.5, 4) node {$\scriptstyle{H_{\lfloor\frac{n-3}{2}\rfloor}(\fS_{n-2})}$} (4.5, 5) node {$\scriptstyle{H_{\lfloor\frac{n-1}{2}\rfloor}(\fS_{n-2})}$};
			\foreach \y in {0,1,2}
			\draw (1.25,\y) node {$\scriptstyle{H_\y(\fS_{n-1})}$} (4.5,\y) node {$\scriptstyle{H_\y(\fS_{n-2})}$} (8.75,\y) node {$\scriptstyle{H_\y(\fS_{0})}$} (11,\y) node {$0$};
			\foreach \y in {0,1,2}
			\foreach \x in {(10.5,\y),(7.8,\y), (0, \y)} 
			\draw[->] \x -- +(-0.3,0);;
			\draw[->] (-.2,4) -- +(-0.3,0);
			\draw[->] (-.2,5) -- +(-0.3,0);
			\draw[->] (-.2,6) -- +(-0.3,0);
			\foreach\x in {0,1,2,4,5}
	        \draw[->] (3,\x) -- node[midway, above] {$\scriptstyle{0}$} +(-0.5,0);
	        \foreach\x in {0,1,2,4}
	        \draw[->] (6.3,\x) -- node[midway, above] {$\scriptstyle{s_\ast}$} node[midway, below] {{\color{red}$\cong$}} +(-0.5,0);
	        \draw[->] (6.3,5) -- node[midway, above] {$\scriptstyle{s_\ast}$} +(-0.5,0);
		\end{tikzpicture}
\caption{The~$E^1$ page of the spectral sequence for the double complex, with differentials labelled as per \autoref{lemma-d-one}.}
\label{fig:spectralsequence}
\end{figure}
    
    Homological stability for the symmetric groups~\cite{N} states that the stabilisation map
    \[
    s_\ast\colon H_i(\fS_{m-1};\t)\longrightarrow H_i(\fS_{m};\t)
    \]
    is an isomorphism in the range~$i\leq\frac{m-1}{2}$. 
    \autoref{lemma-d-one} then tells us that the differential $d^1\colon E^1_{p+1,q}\to E^1_{p,q}$ is an isomorphism in degrees $q\leq \frac{n-p-2}{2}$ when $p$ is even, and is $0$ when $p$ is odd.
    It then follows, exactly as in many proofs of homological stability for the symmetric groups (\cite[IV~3.1]{Maazen}, \cite[Theorem 5.1]{RandalWilliamsConfig}, \cite[Section 3]{Gan}), that from the $E^2$-term onwards, when $q\leq\frac{n-1}{2}$ all remaining sources of differentials into the $(-1)$ column are zero.
    
    Putting this together with the fact that the spectral sequence converges to zero in the range~$p+q\leq \frac{n-3}{2}$ (and noting that~$\frac{n-3}{2}=-1+\frac{n-1}{2}$) it follows that the $(-1)$-column on the~$E^2$-page must be zero in the range~$q\leq \frac{n-1}{2}$. Therefore the~$d^1$-differential into the $(-1)$-column is a surjection in this range.
\end{proof}

\begin{proof}[Proof of \autoref{thma}]
\autoref{thm:BncongSn-1} showed that the map
\[ H_i(\fS_{n-1};\t) \longrightarrow \Tor_i^{\Br_n}(\t,\t)\]
that is induced by the inclusion $R\fS_{n-1}\hookrightarrow \Br_n$ is surjective when $i\leq\frac{n-1}{2}$.
Observe that the inclusion $R\fS_{n-1} \hookrightarrow \Br_n$ factors as $R\fS_{n-1}\hookrightarrow R\fS_n \hookrightarrow \Br_n$, so that the diagram
\[
    \xymatrix{
        H_i(\fS_{n-1};\t) \ar[d]\ar[r] &\Tor_i^{\Br_n}(\t,\t)\\ H_i(\fS_n;\t)\ar[ur]
    }
\]
commutes by functoriality. 
It follows that the map 
\[\iota_\ast H_i(\fS_{n};\t) \longrightarrow \Tor_i^{\Br_n}(\t,\t)\]
induced by the inclusion~$R\fS_n\hookrightarrow \Br_n$ is surjective in the same range $i\leq \frac{n-1}{2}$ as $H_i(\fS_{n-1};\t) \to \Tor_i^{\Br_n}(\t,\t)$.

Together with injectivity of $\iota_\ast$, this implies that 
\[\iota_\ast H_i(\fS_{n};\t) \longrightarrow \Tor_i^{\Br_n}(\t,\t)\]
is an isomorphism if  $i\leq \frac{n-1}{2}$ as required.
\end{proof}

\bibliographystyle{alpha}
\bibliography{repstab.bib}		

\end{document}